\documentclass[11pt]{amsart}

\usepackage{algorithm}
\usepackage{algorithmic}
\usepackage{amsfonts}
\usepackage{amscd}
\usepackage{color}
\usepackage{amsmath}
\usepackage{float}
\usepackage{amsthm}
\usepackage{amssymb}
\usepackage{mathrsfs}
\usepackage{amstext}
\usepackage{todonotes}
\usepackage{graphicx}
\usepackage{tikz}
\usepackage{verbatim}

\usepackage{xcolor}

\newcommand{\BIGOP}[1]{\mathop{\mathchoice%
		{\raise-0.22em\hbox{\huge $#1$}}%
		{\raise-0.05em\hbox{\Large $#1$}}{\hbox{\large $#1$}}{#1}}}

\newcommand{\BIGboxplus}{\mathop{\mathchoice%
		{\raise-0.35em\hbox{\huge $\boxplus$}}%
		{\raise-0.15em\hbox{\Large $\boxplus$}}{\hbox{\large $\boxplus$}}{\boxplus}}}
\numberwithin{equation}{section}

\theoremstyle{plain}
\newtheorem{satz}{Theorem}[section]
\newtheorem{defi}[satz]{Definition}
\newtheorem{cor}[satz]{Corollary}
\newtheorem{lem}[satz]{Lemma}
\newtheorem{prop}[satz]{Proposition}
\newtheorem{rem}[satz]{Remark}

\newcommand{\re}{\ensuremath{\mathbb{R}}}
\newcommand{\N}{\ensuremath{\mathbb{N}}}
\newcommand{\zz}{\ensuremath{\mathbb{Z}}}

\newcommand{\Z}{{\ensuremath{\zz}^d}}

\newcommand{\R}{\ensuremath{{\re}^d}}
\newcommand{\jb}{\boldsymbol{j}}

\newcommand{\x}{\boldsymbol{x}}

\newcommand{\kb}{\boldsymbol{k}}
\newcommand{\lb}{\boldsymbol{l}}
\newcommand{\nb}{\boldsymbol{n}}
\newcommand{\tb}{\boldsymbol{t}}
\newcommand{\Nb}{\boldsymbol{N}}
\newcommand{\Lb}{\boldsymbol{L}}
\newcommand{\Jb}{\boldsymbol{J}}
\newcommand{\Qb}{\boldsymbol{Q}}
\newcommand{\Sb}{\boldsymbol{S}}

\newcommand{\cf}{\ensuremath{\mathcal F}}

\newcommand{\Span}{{\rm span \, }}

\newcommand{\supp}{{\rm supp \, }}

\newcommand{\bproof}{\begin{proof}}
	\newcommand{\eproof}{\end{proof}}

\newlength{\fixboxwidth}
\newcommand{\be}{\begin{equation}}
\newcommand{\ee}{\end{equation}}
\newcommand{\beq}{\begin{eqnarray}}
\newcommand{\beqq}{\begin{eqnarray*}}
\newcommand{\eeq}{\end{eqnarray}}
\newcommand{\eeqq}{\end{eqnarray*}}

\title{}
\author{}

\begin{document}
	\allowdisplaybreaks
	
	\title
	[Adaptive sampling recovery]
	{Adaptive sampling recovery of functions with higher mixed regularity}

\author[Nadiia Derevianko]{Nadiia Derevianko$^{\lowercase{\rm a, b}}$}
\address{Fakult\"at f\"ur Mathematik\\ Technische Universit\"at  Chemnitz\\09107 Chemnitz, Germany, Institute of Mathematics of NAS of Ukraine, Tereshchenkivska st. 3, 01601 Kyiv-4, Ukraine}
\email{nadiia.derevianko@mathematik.tu-chemnitz.de}

\author[Tino Ullrich]{Tino Ullrich$^{\lowercase{\rm a}}$}
\address{Fakult\"at f\"ur Mathematik\\ Technische Universit\"at  Chemnitz\\09107 Chemnitz, Germany}
\email{tino.ullrich@mathematik.tu-chemnitz.de}

\thanks{$^{\lowercase{\rm a}}$Fakult\"at f\"ur Mathematik, Technische Universit\"at  Chemnitz, 09107 Chemnitz, Germany}

\thanks{$^{\lowercase{\rm b}}$Institute of Mathematics of NAS of Ukraine, Tereshchenkivska st. 3, 01601 Kyiv-4, Ukraine}

\begin{abstract}
We tensorize the Faber spline system from \cite{DU2019} to prove sequence space isomorphisms for multivariate function spaces with higher mixed regularity. The respective basis coefficients are local linear combinations of discrete function values similar as for the classical Faber Schauder system. This allows for a sparse representation of the function using a truncated series expansion by only storing discrete (finite) set of function values. The set of nodes where the function values are taken depends on the respective function in a non-linear way. Indeed, if we choose the basis functions adaptively it requires significantly less function values to represent the initial function up to accuracy $\varepsilon>0$ (say in $L_\infty$) compared to hyperbolic cross projections. In addition, due to the higher regularity of the Faber splines we overcome the (mixed) smoothness restriction $r <2$ and benefit from higher mixed regularity of the function. As a byproduct we present the solution of Problem 3.13 in Triebel's monograph \cite{Tr2012} for the multivariate setting.
\end{abstract}

\maketitle

\section{Introduction}
This paper is a continuation of the research that was started in \cite{DU2019}, where we constructed a higher order Faber spline basis for the representation of univariate functions from Besov-Lizorkin-Triebel spaces. Compared to wavelet isomorphisms, the coefficient functionals in a Faber spline series expansion are given by a (local) linear combination of discrete function values. In the monographs \cite{Tr2010, Tr2012} Triebel used tensor products of the classical Faber-Schauder basis to obtain characterization for bivariate functions. Due to the limited regularity of the hat functions one is not able to benefit from higher smoothness of the respective functions. This leads to a saturation rate $n^{-2}$ for the approximation error. Here we present the corresponding modifications for the multivariate situation to resolve functions with higher mixed Sobolev-Besov regularity. The characterizations are used to sparsely represent a multivariate function by a truncated tensorized Faber spline series. As already observed by Byrenheid in \cite{Glen2018} one needs to store significantly less discrete information of the function compared to a hyperbolic cross projection. In fact, to benefit from a non-linear (adaptive) approximation framework we also need to establish characterizations in the quasi-Banach setting. It is well-known (see for instance \cite{Pie80, DeVore1998, TW09, HS2010, HS2011, HS2012}) that Besov spaces with mixed smoothness $S^r_{p,p}B$, where $p<1$, serve as nonlinear approximation spaces in $L_2$ with respect to a hyperbolic  wavelet dictionary since the spaces are isomorphic to $\ell_p$ if $r = 1/p-1/2$.

In our paper \cite{DU2019} we study this problem for the univariate situation.  We construct a higher order Faber spline basis that allows to get sampling discretizations of Besov function spaces $B^r_{p,\theta}$ with higher smoothness $1/p<r<\min\{2m-1+1/p,2m\}$, $m \in \N$ and $m\geq 2$ (see Subsection 2.1 for the definition of this basis). The main tool for the construction of higher order Faber splines are the celebrated Chui-Wang wavelets \cite{Chui1992}. They serve as the initial point. We apply a ``lifting procedure'' to the dual Chui-Wang wavelet via a certain integral operator similar as the Faber-Schauder system evolves from integrating the Haar system. With this technique we are able to generate Faber splines of arbitrary high (but limited) smoothness. Although the new basis functions are supported on the real line they are very well localized (exponentially decaying) and the main parts are concentrated on a segment. Note also that for the periodic case very well time-localized (polynomial decay) basis functions were constructed in \cite{PS2001} and \cite{DMP2019} for one- and two-dimensional cases. The corresponding characterization of Besov spaces (for the univariate case so far) was obtained in \cite{DMP2017}.

The sampling characterization of mixed Besov spaces via Faber-Schauder coefficients for dimension $d>1$ has been established in \cite{Glen2018}. It is known that coefficients in a series expansion with respect to the tensor product Faber-Schauder basis are represented via second order mixed difference  $$d_{\boldsymbol{j},\boldsymbol{k}}:=
\Big(-\frac{1}{2}\Big)^{|e(\jb)|}\Delta^{2,e(\jb)}_{2^{-(\jb+1)}}f(\x_{\jb,\kb})\quad,\quad \boldsymbol{j}\in \N_{-1}^d, \boldsymbol{k}\in \Z\,,$$
where we put $e(\boldsymbol{j}):=\{i\in [d]: \, j_i\geq 0\}$ for $\boldsymbol{j}\in \N_{-1}^d$ and $\Delta^{2,e(\jb)}_{2^{-(\jb+1)}}f:= \prod_{i \in e(\jb)}
	\Delta^{2,i}_{2^{-j_i-1}}\,,$
where the latter denotes the second order difference operator which is applied to the $i$-th variable of the function $f$.

The interesting question was if this extends to a higher order framework. This question is answered in this paper. Indeed, if we tensorize the Faber splines from \cite{DU2019} we obtain a linear combination of $2m$-th order mixed differences ($m \geq 1$)
$$
\lambda_{2m;\jb,\kb}(f)=\sum\limits_{\boldsymbol{l}\in \iota_m} \Big( \prod\limits_{i\in e(\jb)} (-1)^{l_i} \Big) N_{2m}(\boldsymbol{l}+\boldsymbol{1}) \Delta^{2m,e(\boldsymbol{j})}_{2^{-\boldsymbol{j}-1}}f(\boldsymbol{x}_{\boldsymbol{j};\boldsymbol{k},\boldsymbol{l}}) \quad,\quad \jb \in \N_{-1}^d, \kb \in \Z\,,
$$
where $\iota_m:=\{\boldsymbol{l}\in \Z: \, l_i=0, \ldots, 2m-2, \, i=1,\ldots,d \}$ and $\boldsymbol{N}_m$ is the $m$-th order tensorized $B$-spline (see Section 2 for details) and $\boldsymbol{x}_{\jb;\boldsymbol{k},\boldsymbol{l}}:=(x_{j_1;k_1,l_1},...,x_{j_d;k_d,l_d})$,
where
$$
x_{j;k,l}:=
\begin{cases}
k, & j=-1,\\
\frac{2k+l}{2^{j+1}}, & j\in \N_0.
\end{cases}
$$

 In the case of approximation by linear sampling algorithms such as for example Smolyak sparse grid operators the set of samples is fixed in advance for the whole class of functions. In this paper we consider adaptive sampling approximation where the set of sample points is chosen for each particular function differently. Namely, we work with the quantity of best $n$-term approximation with respect to this new Faber spline basis. A detailed overview on the history of this subject with regard to different dictionaries is given in \cite[Chapt.\ 10]{DTU-book}. We mention here only some papers of Hansen and Sickel \cite{HS2010}, \cite{HS2011}, \cite{HS2012}, D. D\~ung \cite{Dung2000, Dung2001}, Temlyakov \cite{Temlyakov1998, Temlyakov2000}, V.~Romanyuk \cite{Romanyuk2016}, \cite{Roman2016} and Bazarkhanov \cite{Baz2014}, where the authors considered wavelet type dictionaries. Best $n$-term trigonometric approximation was studied by Temlyakov \cite{Tem2015}, \cite{Tem2017}, A.~Romanyuk \cite{Rom1993}, \cite{Rom1995}, \cite{Rom2003}, \cite{Rom2007}, A.~Romanyuk and V.~Romanyuk \cite{RR2010}, Stasyuk \cite{Stas2014}, \cite{Stas2016} and others. Probably closest to us are the papers D. D\~ung \cite{DD2009,DD2013}, where the univariate and isotropic framework is considered. Here we consider classes of multivariate functions with mixed smoothness.

We study best $n$-term approximation with respect to the tensorized Faber spline basis $\mathcal{B}^d_{2m}$ constructed in this paper. The main advantage over classical wavelet dictionaries is the fact that the aggregate of best $n$-term approximation is constructed by using only finitely many discrete function values rather than wavelet inner products. The following order estimates are obtained below for function space embedding on a compact $K \subset \R$. In the so-called case of ``small smoothness'' we obtain
$$
		\sigma_n(S_{p,\theta}^rB(K),\mathcal{B}_{2m}^d)_{q}\lesssim n^{-r}
$$
for $0<p< q<\infty$ and smoothness $1/p<r<\min\{1/\theta-1/\min\{1,q\},2m\}$, whereas in the case of ``large smoothness''  $\max\{1/p,1/\theta-1/\max\{q,1\}\}<r<2m$ we find
$$
 \sigma_n(S_{p,\theta}^rB(K),\mathcal{B}_{2m}^d)_{q}\lesssim n^{-r} (\log n)^{(d-1)(r-1/\theta+1)}.
$$
Note that in case of wavelet type dictionaries and $1<q<\infty$ there is a slight improvement in the power of the $\log$, namely $r-1/\theta+1/2$.
These results complement the corresponding results for the Faber-Schauder system, see \cite{Glen2018}. Similarly to the Faber-Schauder case the estimates for best $n$-term approximation are realized by a so-called level-wise greedy algorithm that uses $C(d,m)n$  coefficients. Due to the construction the number of function values that have to be stored scales similarly.

\textbf{Outline.} This paper has the following structure. In Subsection 2.1 we give the definition of the 1D higher order Faber spline basis that was constructed in \cite{DU2019}. In Subsections 2.2 and 2.3 we define the corresponding basis for the multivariate setting and prove uniform convergence in $C_0(\R)$.  In Section 3 we prove the sampling characterization for Besov-Triebel-Lizorkin spaces of mixed smoothness. The main theorem is formulated in Subsection 3.4. In Section 2 we prove the sampling characterization of Besov-Triebel-Lizorkin spaces of mixed smoothness. The main theorem is formulated in Subsection 2.4. In Subsection 2.1 we give definition of 1D higher order Faber spline basis that was constructed in \cite{DU2019}.  In Section 4 we consider best $n$-term approximation of Besov-Triebel-Lizorkin spaces of with respect to the higher order Faber spline basis. The main results are presented in Subsections 4.1 and 4.2. The short Subsection 4.3 refers to the constructiveness of the algorithm we use.  Finally, definitions of functions spaces and some auxiliary results from harmonic analysis are put to the Appendix A. Examples of the construction of higher order Faber splines for 1D are presented in Appendix~B.

\textbf{Notation.} As usual $\N$, $\zz$ and $\re$ are reserved for natural, integer and real numbers. Then let $\N_0:=\N\cup \{0\}$, $\N_{-1}:=\N\cup \{0,-1\}$, $\zz_+:=\{k\in \zz: k\geq0\}$ and $\re_+:=\{x\in \re: x\geq0\}$. With letter $d$ we indicate dimension of spaces $\N^d$, $\Z$, $\R$ and etc. By symbols $\x$, $\tb$ we denote elements from vector spaces $\N^d$, $\Z$, $\R$. $|\x|$ means the following vector $|\x|=(|x_1|,\ldots,|x_d|)$ and by $\x\cdot \tb$ we denote a scaling product of two vectors $\x$ and $\tb$, i.e. $\x\cdot \tb=x_1t_1+\ldots+x_dt_d$. By $|\boldsymbol{x}|_1$ we mean $|\boldsymbol{x}|_1=x_1+\ldots+x_d$. For $a\in\re$ we denote by $a_+$ the number $a_+:=\max\{a,0\}$. For two non negative quantities $a$ and $b$ we write $a \lesssim b$ if there exists a positive constant $c$ that does not depend on one of the parameters known from the context such that $a\leq c \, b$. We write $a\asymp b$ if $a \lesssim b$ and $b \lesssim a$. Let $C(\R)$ be the space of continuous functions on $\R$ with the usual supremum norm, $C_0(\R)$ be the space of compactly supported continuous functions on $\R$. By $L_p(\R)$, $0<p\leq \infty$ as usual we denote the space of Lebesgue measurable functions with the finite norm
$$
\|f\|_p:=
\begin{cases}
\Big(\int\limits_{\R} |f(\x)|^p d\x \Big)^{1/p}, & 0<p<\infty, \\
\mathrm{ess}\sup\limits_{\x \in \R} |f(\x)|, & p=\infty.
\end{cases}
$$

\section{A higher order tensorized Faber spline basis}

In this section we present construction of multivariate higher order  Faber splines and prove sampling discretization of Besov and Lizorkin-Triebel spaces of functions with mixed smoothness.

First we give definition of this new basis for the univariate case. The main ideas of the construction for 1D are presented in \cite{DU2019} and further we give only necessary definitions.

\subsection{Definition of univariate basis functions}
Let $N_m$, $m \in \N$, be the $m$-th order B-spline with knots at $\zz$ defined by
$$
N_m(x)=(N_{m-1}*N_1)(x)=\int \limits_0^1 N_{m-1}(x-t) dt,
$$
where $N_1=\mathcal{X}_{[0,1)}$.
Let further $j,k \in \zz$ and $N_{m;j,k}:=N_m(2^j \cdot-k)$. It is well-known that the system
$$
V_j:=\text{span}\{N_{m;j,k}: k \in \zz \}.
$$
constitutes a multiresolution analysis of $L_2(\re)$.
In \cite{Chui1992} it was also proved that there is a compactly supported wavelet that generates the wavelet spaces $W_j$, usually defined as $W_j=V_{j+1}\ominus V_j, \, j \in \zz$, represented by the formula
$$
\psi_m(x)=\dfrac{1}{2^{m-1}}\sum\limits_{l=0}^{2m-2}(-1)^lN_{2m}(l+1)N_{2m}^{(m)}(2x-l).
$$
It means that $W_j=\mathrm{span}\{\psi_m(2^{j}\cdot -k), \, k \in \zz\}$. It is clear that $\supp N_m=[0,m]$ and $\supp \psi_m=[0,2m-1]$.

By $\psi_m^*$ we denote the dual wavelet of $\psi_m$, and by $N_m^*$ dual of $N_m$.
In \cite{DU2019} we prove an explicit representation for $\psi_m^*$.

\begin{satz} \cite{DU2019} \label{dual-wav-repr}
	The dual wavelet $\psi_m^*$ can be represented as
	$$
	\psi_m^*(x)=\sum\limits_{n\in \zz} a_n^{(m)}\psi_m(x-n),
	$$
	where the coefficients $a_n^{(m)}$ are exponentially decaying and can be computed precisely by \cite[Algorithm 1]{DU2019} (see also the Appendix for $m=2$ and $m=3$).
\end{satz}
Further we define the cardinal spline function
\begin{equation}\label{card-sp}
L^m(x):=\sum \limits_{n \in \zz} c_n^{(m)}N_m(x+m/2-n).
\end{equation}
with precisely given exponentially decaying coefficients (see \cite{Chuibook} for details). By using the $m$-th order biorthogonal Chui-Wang wavelets and Theorem \ref{dual-wav-repr} we define the following $2m$-th order Faber spline basis for $j \in \N_0$ as
\begin{align}
s_{2m;j,k}(x)&=2^{mj}\int\limits_{-\infty}^x \frac{\psi_{m;j,k}^*(t)}{(m-1)!}(x-t)^{m-1} dt \notag\\
&=2^{mj} \sum \limits_{n \in \zz} a_n^{(m)} \int\limits_{-\infty}^x \frac{\psi_{m;j,k}(t)}{(m-1)!}(x-t)^{m-1} dt \label{b-basis}
\end{align}
and $s_{2m;-1,k}(x):=L^{2m}(x-k)$. Plotted examples of tensorized versions of these basis functions $s_{4;0,0}$ and $s_{6;0,0}$ are given in Figures \ref{b-basis-b2} and \ref{b-basis-b6}. For comparison the Faber-Schauder basis function is presented on Figure \ref{b-basis0}.

We would also like to mention a related univariate  construction in the papers \cite{Wang96, Wang95}. For further
details in this regard see \cite[Remark 3.6]{DU2019}.

\begin{figure}[h!]
	\center{\includegraphics[width=0.55\linewidth]{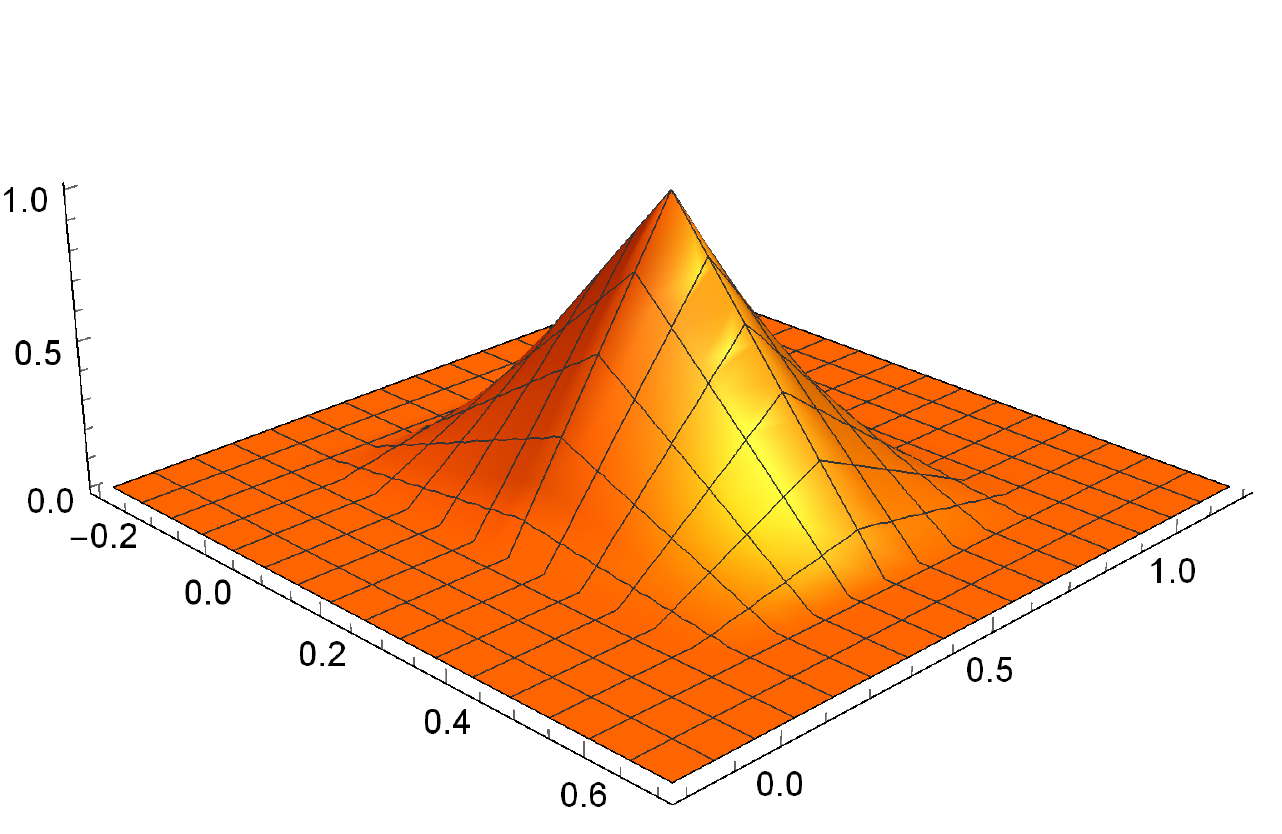}}
	\caption{The Faber-Schauder basis function with the dilation index $\boldsymbol{j}=(4,0)$.}
	\label{b-basis0}
\end{figure}


\subsection{Construction of multivariate basis functions}
 For $f \in C_0(\re)$ the following series converges uniformly
$$
f=\sum\limits_{j\in\N_{-1}}\sum\limits_{k\in \zz} \lambda_{2m;j,k}(f) s_{2m;j,k},
$$
where the coefficients $\lambda_{2m;j,k}(f)$ are defined as follows $\lambda^{(m)}_{-1,k}(f):=f(k)$ and for $j\in \N_0$
$$
\lambda_{2m;j,k}(f)=\sum\limits_{l=0}^{2m-2}(-1)^{l}N_{2m}(l+1)\Delta_{2^{-j-1}}^{2m} f\Big( \frac{2k+l}{2^{j+1}}\Big),
$$
where by $\Delta_{h}^{2m}f(x)$ we denote a $2m$-th order difference of a function $f$ with step $h$ at the point~$x$ (see Appendix A.3 for a definition).

Now let a function $f \in C_0(\R)$. Recall that for $\boldsymbol{j}\in \N_{-1}^d$, $\boldsymbol{k} \in \Z$, $e(\boldsymbol{j}):=\{i\in [d]: \, j_i\geq 0\}$. For $m\in \N$ we define a set $\iota_m:=\{\boldsymbol{l}\in \Z: \, l_i=0, \ldots, 2m-2, \, i=1,\ldots,d \}$ and let $\boldsymbol{x}_{\boldsymbol{j};\boldsymbol{k},\boldsymbol{l}}=(x_{j_1;k_1,l_1},...,x_{j_d;k_d,l_d})$,
where
$$
x_{j;k,l}:=
\begin{cases}
k, & j=-1,\\
\frac{2k+l}{2^{j+1}}, & j\in \N_0.
\end{cases}
$$
We fixed variables $x_2,...,x_d$ and consider the expansion with respect to a variable $x_1$
\begin{align*}
f(x_1,...,x_d)&=\sum\limits_{k_1\in \zz} f(k_1,x_2,...,x_d) s_{2m;-1,k_1}(x_1) \\
&+\sum\limits_{j_1\in\N_{0}}\sum\limits_{k_1\in \zz} \left(\sum\limits_{l=0}^{2m-2}(-1)^{l}N_{2m}(l+1)\Delta_{2^{-j_1-1}}^{2m} f\Big( \frac{2k_1+l}{2^{j_1+1}},x_2,...,x_d\Big)\right) s_{2m;j_1,k_1}(x_1).
\end{align*}
Further we expand functions $f(k_1,x_2,...,x_d)$ and $f\Big( \frac{2k_1+l}{2^{j_1+1}},x_2,...,x_d\Big)$ with respect to a variable $x_2$. Proceeding in this way for each variable we obtain the following pointwise expansion
\begin{equation}\label{exp-mul}
f=\sum\limits_{\boldsymbol{j}\in\N_{-1}^d}\sum\limits_{\boldsymbol{k} \in \Z} \lambda_{2m;\jb,\kb}(f) \boldsymbol{s}_{2m;\jb,\kb},
\end{equation}
where basis functions $\boldsymbol{s}_{2m;\jb,\kb}$ are defined as follows
$$
\boldsymbol{s}_{2m;\jb,\kb}(\boldsymbol{x})=s_{2m;j_1,k_1}(x_1)\cdot ... \cdot s_{2m;j_d,k_d}(x_d),
$$
and coefficients
$$
\lambda_{2m;\jb,\kb}(f)=\sum\limits_{\boldsymbol{l}\in \iota_m} \Big( \prod\limits_{i\in e(\jb)} (-1)^{l_i} \Big) N_{2m}(\boldsymbol{l}+\boldsymbol{1}) \Delta^{2m,e(\boldsymbol{j})}_{2^{-\boldsymbol{j}-1}}f(\boldsymbol{x}_{\boldsymbol{j};\boldsymbol{k},\boldsymbol{l}}),
$$
where $\Nb_{2m}(\boldsymbol{l}+\boldsymbol{1})=\prod\limits_{i=1}^d N_{2m}(l_i+1)$ and $\Delta^{m,e}_{\boldsymbol{h}}$ is $2m$-th order mixed difference operator acting in the directions contained in $e$ (see Appendix A.3 for definition). Further we will use the following notation for the sequence of coefficients $\lambda^{(m)}(f):=\{\lambda_{2m;\jb,\kb}(f): \, \boldsymbol{j}\in \N_{-1}^d$, $\boldsymbol{k} \in \Z \}$.

\subsection{Uniform convergence in the space $C_0(\R)$}
Further we show that series (\ref{exp-mul}) converges uniformly. First we give some necessary definitions. The univariate cardinal spline function is defined by (\ref{card-sp}).
Then for dimension $d>1$,
$$
\Lb^m(\boldsymbol{x})=\prod\limits_{i=1}^d L^m(x_i), \ \ \ \boldsymbol{x}=(x_1,\ldots,x_d),
$$
or we can also write
$$
\Lb^m(\boldsymbol{x})=\sum \limits_{\boldsymbol{n} \in \Z} c_{\boldsymbol{n}}^{(m)}\Nb_m(\boldsymbol{x}+\boldsymbol{m}/2-\boldsymbol{n}),
$$
where $c_{\boldsymbol{n}}^{(m)}=\prod\limits_{i=1}^d c_{n_i}^{(m)}$ and $\boldsymbol{m}/2=(m/2,...,m/2)$. It is clear that
 $\Lb^m(\boldsymbol{j})=\delta_{\boldsymbol{j},0}$ for $\boldsymbol{j}\in \Z$.
Further we define the fundamental spline interpolating operator $\Jb^m$ as
$$
(\Jb^mf)(\boldsymbol{x}):=\sum \limits_{\boldsymbol{n} \in \Z} f(\boldsymbol{n})\Lb^m(\boldsymbol{x}-\boldsymbol{n})
$$
and its scaled version $\Jb^m_N$, $N \in \N$, as
$$
(\Jb^m_Nf)(\boldsymbol{x})=\sum \limits_{\boldsymbol{n} \in \Z} f(2^{-N} n_1,...,2^{-N} n_d) \Lb^m(2^N \boldsymbol{x}-\boldsymbol{n})
$$
with the following interpolating property $(\Jb^m_Nf)(\boldsymbol{n}/2^N)=f(\boldsymbol{n}/2^N)$ for $\boldsymbol{n}\in \Z$.

Let $f_i$ be a univariate function defined as
\begin{equation} \label{1d}
 f_i(x_i)=\sum \limits_{n_i \in \zz} a_{n_i} N_m(2^Nx_i-n_i), \ \ \{a_{n_i}\}_{n_i\in \zz} \in l_1.
\end{equation}
By $V_N^m$ we denote the space of multivariate functions of the form
$$
f(\boldsymbol{x})=\prod\limits_{i=1}^d f_i(x_i),
$$
where $f_i$ is defined by (\ref{1d}).

\begin{figure}[h!]
	\center{\includegraphics[width=0.55\linewidth]{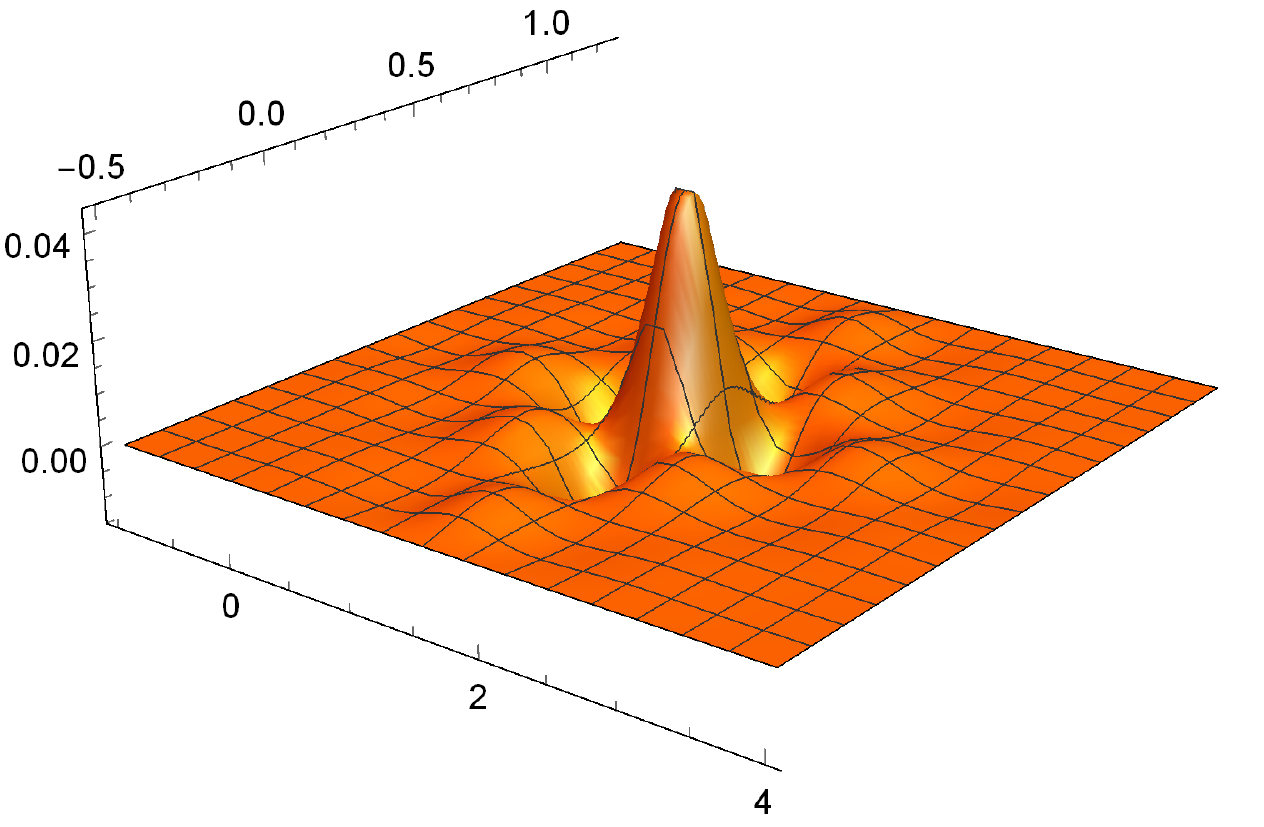}}
	\caption{The basis function $-\boldsymbol{s}_{4;\boldsymbol{j},\boldsymbol{0}}$ for $\boldsymbol{j}=(0,4)$.}
	\label{b-basis-b2}
\end{figure}

Further we state two lemmas.
\begin{lem} \label{recov1}
	Every $f \in V_N^m$ can be reproduced by the fundamental spline interpolating operator $\Jb_N^m$, i.e.
	$$
	(\Jb_N^mf)(\boldsymbol{x})=f(\boldsymbol{x}), \ \forall \boldsymbol{x} \in \R.
	$$
\end{lem}
\begin{proof}
The proof is similar to the one-dimensional case (see Lemma 3.2 in \cite{DU2019}). We use the fact that $\Lb^m(\boldsymbol{j})=\delta_{\boldsymbol{j},0}$ for $\boldsymbol{j}\in \Z$.
\end{proof}	

Using the expansion (\ref{exp-mul}) we define the operator $\Sb_N^{m}$
$$
\Sb_N^{m}f=\sum\limits_{|\boldsymbol{j}|_{\infty}<N}\sum\limits_{\boldsymbol{k} \in \Z} \lambda_{2m;\jb,\kb}(f) \boldsymbol{s}_{2m;\boldsymbol{j},\boldsymbol{k}}.
$$	
\begin{lem} \label{recov2}
	Every $f \in V_N^m$ can be reproduced by the  operator $\Sb_N^{m}$, i.e.
	$$
	(\Sb_N^{m}f)(\boldsymbol{x})=f(\boldsymbol{x}), \ \forall \boldsymbol{x} \in \R.
	$$
\end{lem}
\begin{proof}
The proof follows from the fact $\Sb_N^{m}f=f$ for $d=1$ (see Lemma 3.1 \cite{DU2019}) and tensor structure of functions from the space $V_N^m$.
\end{proof}	

As a consequence from this two lemmas we get the following
\begin{lem} \label{recov3}
	For every function $f \in C_0(\R)$ we have $\Sb_N^{m}f \equiv \Jb_N^mf$.
\end{lem}
\begin{proof}
	Let $f \in C_0(\R)$ and $\boldsymbol{x} \in \supp f$. Since $\Jb_N^mf \in V_N^m$ according to Lemma \ref{recov1} we have that $\Sb_N^{m}(\Jb_N^mf)(\boldsymbol{x})=\Jb_N^mf(\boldsymbol{x})$. On the other hand, since $\Jb_N^mf(\boldsymbol{k}/2^N)=f(\boldsymbol{k}/2^N)$, $\boldsymbol{k}  \in \Z \cap \supp f$, according to definition of $\Sb_N^{m}$ we get that $\Sb_N^{m}(\Jb_N^mf)(\boldsymbol{x})=\Sb_N^{m}f(\boldsymbol{x})$.  	
\end{proof}

\begin{lem} \label{norm-lemma}
For each $m \in \N$
	\begin{equation} \label{jm}
	\|\Jb^mf\|_{\infty \rightarrow \infty} \leq K(m),
	\end{equation}
	where the constant $K(m)$ depends only on $m$.
\end{lem}
\begin{proof}
For $\boldsymbol{x} \in \R$ we have
\begin{align*}
|\Jb^mf(\boldsymbol{x})|&=\bigg|\sum \limits_{\boldsymbol{n} \in \Z} f(\boldsymbol{n})\Lb^m(\boldsymbol{x}-\boldsymbol{n}) \bigg| \leq \|f\|_{\infty} \sum \limits_{\boldsymbol{n} \in \Z} |\Lb^m(\boldsymbol{x}-\boldsymbol{n})|\\
& =  \|f\|_{\infty} \sum \limits_{\boldsymbol{n} \in \Z}  \bigg| \sum \limits_{\boldsymbol{k} \in \Z} c_{\boldsymbol{k}}^{(m)}\Nb_m(\boldsymbol{x}-\boldsymbol{n}-\boldsymbol{k}+\boldsymbol{m}/2)\bigg|\\
& \leq \|f\|_{\infty} \sum \limits_{\boldsymbol{k} \in \Z} |c_{\boldsymbol{k}}^{(m)}|  \sum \limits_{\boldsymbol{n} \in \Z}  |\Nb_m(\boldsymbol{x}-\boldsymbol{n}-\boldsymbol{k}+\boldsymbol{m}/2)|
\end{align*}
Since $\Nb_m$ is compactly supported then the sum $\sum_{\boldsymbol{n} \in \Z}  |\Nb_m(\boldsymbol{x}-\boldsymbol{n}-\boldsymbol{k}+\boldsymbol{m}/2)|$ is finite and depends only on $m$. Since coefficients $c_{\boldsymbol{k}}^{(m)}$ decay exponentially the sum $\sum_{\boldsymbol{k} \in \Z} |c_{\boldsymbol{k}}^{(m)}|$ is also finite and depends only on $m$. Taking this into account we get (\ref{jm}).
\end{proof}

\begin{satz}\label{conv-c}
	For a  function $f \in C_0(\R)$, we have that
	\begin{equation}\label{C-conv}
	\lim \limits_{N\rightarrow \infty} \|f - \Sb_N^{m}f\|_{\infty}=0.
	\end{equation}
\end{satz}
\begin{proof}
	By $Q_N^m$ we denote a univariate quasi-interpolation operator (see \cite{DD2016} for details)
	$$
	Q_{N}^m(x)=\sum\limits_{k \in \nu(N)} a_k^{(N)}(f) N_m(2^N x-k),
	$$
	where the set $\nu(N)$ is  finite and the functional $a_k^{(N)}(f)$ is defined by using finite number of function values.
	 	We define the operator $\Qb_{N}^m(\boldsymbol{x})$ for $\boldsymbol{x} \in \R$ as
	$$
	\Qb_{N}^m(\boldsymbol{x})=\prod\limits_{i=1}^d Q_{N}^m(x_i).
	$$	
	
	Since according to Lemma \ref{recov3} we have that $\Sb_N^{m}=\Jb_N^m$, we can write
	\begin{align*}
	\|f - \Sb_N^{m}f\|_{\infty}&=\|f - \Jb_N^mf\|_{\infty}\leq \|f - \Qb_{N}^mf\|_{\infty}+ \| \Qb_{N}^mf-\Jb_N^mf\|_{\infty}\\
	&=\|f - \Qb_{N}^mf\|_{\infty} + \|\Jb_N^m\left(f - \Qb_{N}^mf \right)\|_{\infty}\\
	& \leq \left(1+ \|\Jb_N^m\|_{\infty \rightarrow \infty}\right) \|f - \Qb_{N}^mf\|_{\infty}.
	\end{align*}
	We used that $\Qb_{N}^mf \in V_N^m$ and according to Lemma \ref{recov1} $\Jb_N^m(\Qb_{N}^mf)=\Qb_{N}^mf$.
	
	Further we use the facts that $\|f - \Qb_{N}^mf\|_{\infty} \rightarrow 0$ if $N\rightarrow \infty$ (see \cite{DD2016}) and the norm $\|\Jb_N^m\|_{\infty \rightarrow \infty}$ is bounded according to Lemma \ref{norm-lemma}. It implies  (\ref{C-conv}).
\end{proof}

\begin{figure}[h!]
	\center{\includegraphics[width=0.55\linewidth]{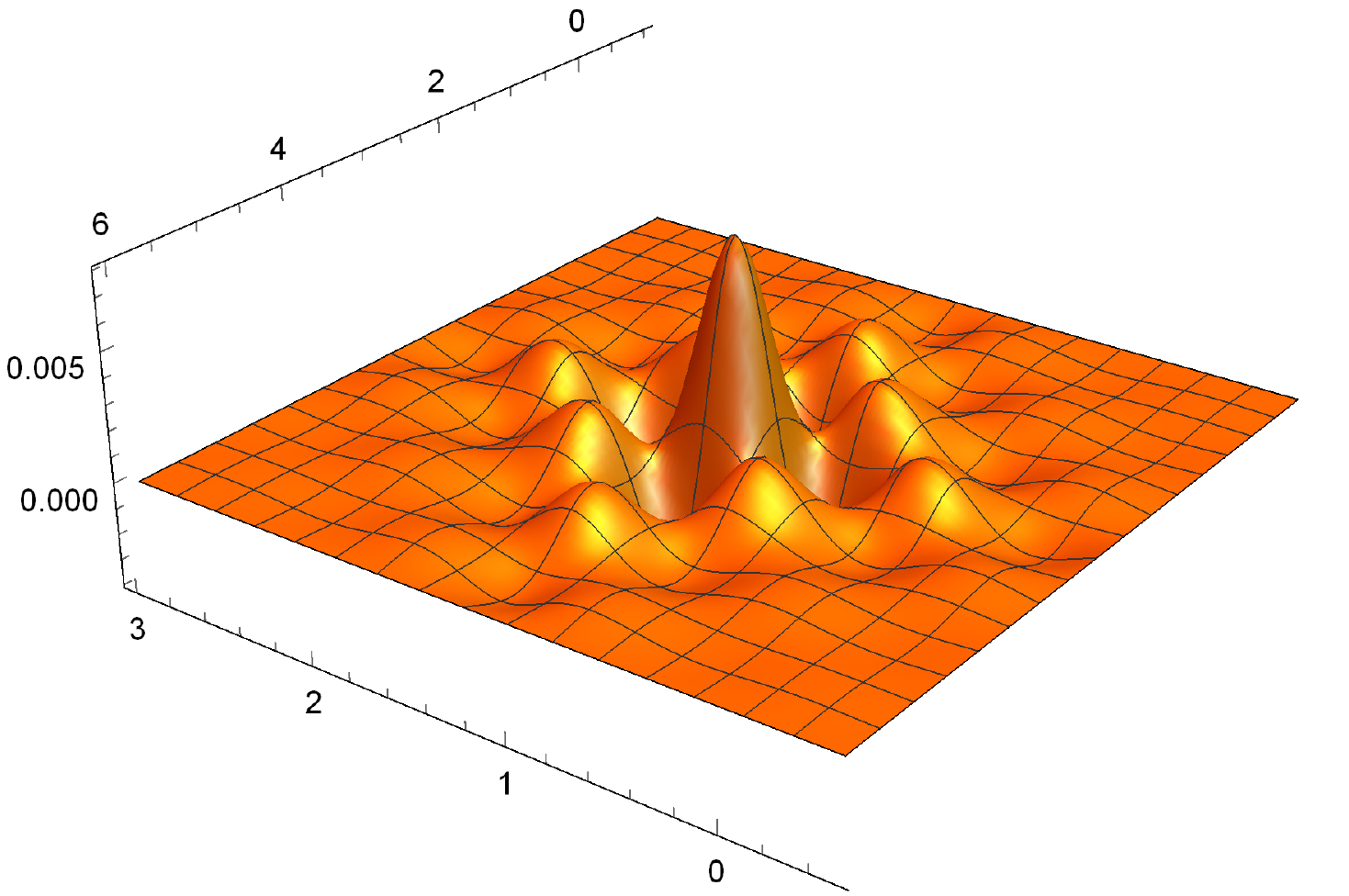}}
	\caption{The basis function $\boldsymbol{s}_{6;\boldsymbol{j},\boldsymbol{0}}$ where $\boldsymbol{j}=(2,0)$.}
	\label{b-basis-b6}
\end{figure}

Since for 1D the basis functions $s_{2m;j,k}$ are defined by
$$
s_{2m;j,k}= \sum \limits_{n \in \zz} a_n^{(m)} v_{2m;j,k+n},
$$
where $v_{2m;j,l}$ is defined for $j \in \N_0$ as
$$
v_{2m;j,l}(x):= 2^{mj}  \int\limits_{-\infty}^x \frac{\psi_{m;j,l}(t)}{(m-1)!}(x-t)^{m-1} dt
$$
and $v_{2m;-1,l}(x):=N_{2m}(x+m-l)$, it is clear that $v_{2m;j,l}$ are supported on the segment $[2^{-j}l,2^{-j}(l+2m-1)]$ and $\supp \, v_{2m;-1,l}=[-m+l,m+l]$. We denote a segment $K^{(m)}_{j,l}$
$$
K^{(m)}_{j,l}:= \begin{cases}
[2^{-j}l,2^{-j}(l+2m-1)], & j \in \N_0;\\
[-m+l,m+l], & j=-1.
\end{cases}
$$
For the multivariate case we will have
$$
\boldsymbol{s}_{2m;\jb,\kb}= \sum \limits_{\nb \in \Z} a_{\boldsymbol{n}}^{(m,e(\jb))} \boldsymbol{v}_{2m;\boldsymbol{j},\boldsymbol{k}+\nb},
$$
where $a_{\boldsymbol{n}}^{(m,e(\jb))}=\prod_{i=1}^d a_{n_i}^{(m)}$, where $a_{n_i}^{(m)}$ are given in Theorem \ref{dual-wav-repr} and \eqref{card-sp}. To be more precise, if $a_{n_i}^{(m)}=a_{n_i}^{(m)}$ if $i \in e(\jb)$ and $a_{n_i}^{(m)}=c_{n_i}^{(m)}$ if $i \not\in e(\jb)$. The function $\boldsymbol{v}_{2m;\boldsymbol{j},\boldsymbol{l}}$ is supported on the cube $K_{\jb,\lb}^{(m)}=\prod_{i=1}^dK^{(m)}_{j_i,l_i}$.

\section{Sampling isomorphisms in mixed smoothness spaces}
The main goal of this subsection is to prove sequence space isomorphisms with respect to the tensorized Faber spline basis.

\begin{satz}\label{charact-m} \begin{itemize}
		\item [(i)]
		Let $0< p,\theta\leq \infty$,  $p>1/2m$ and $1/p<r<\min\{2m-1+1/p,2m\}$ for $m\in \N$ and $m\geq 2$. Then every compactly supported $f \in S_{p,\theta}^rB$ can be represented by the series (\ref{exp-mul}), which is convergent unconditionally in the space $S_{p,\theta}^{r-\varepsilon}B$  for every $\varepsilon>0$. If $\max\{p,\theta\}<\infty$ we have unconditional convergence in the space $S_{p,\theta}^{r}B$.  Moreover, the following norms are equivalent
		$$
		\|\lambda^{(m)}(f) \|_{s_{p,\theta}^r b} \asymp \|f\|_{S_{p,\theta}^rB}.
		$$
		\item [(ii)] 	Let  $1/2m< p,\theta\leq \infty$, $p\neq \infty$, and $\max\{1/p,1/\theta\}<r<2m-1$ for $m\in \N$ and $m\geq 2$. Then every compactly supported $f \in S_{p,\theta}^rF$ can be represented by the series (\ref{exp-mul}), which is convergent unconditionally in the space $S_{p,\theta}^{r-\varepsilon}F$  for every $\varepsilon>0$. If $\theta<\infty$ we have unconditional convergence in the space $S_{p,\theta}^{r}F$.  Moreover, the following norms are equivalent
		$$
		\|\lambda^{(m)}(f) \|_{s_{p,\theta}^rf} \asymp \|f\|_{S_{p,\theta}^rF}.
		$$
	\end{itemize}
\end{satz}

\begin{rem}\label{rem-best-3}
For the one-dimensional setting these results were obtained in \cite{DU2019}. The case $m=1$ corresponds to the classical Faber-Schauder basis and the corresponding characterization for $d=1$ was obtained in \cite{Tr2010} and for $d>1$ this result was extended in \cite{Glen2018}. Note again that our results complement results from \cite{Tr2010} and \cite{Glen2018}, but do not include them.
\end{rem}

\begin{rem}\label{rem-best-4}
We would also like to mention about sharpness of the restrictions for the smoothness parameter $r$ in Theorem \ref{charact-m}. First we consider B-case. Lower restriction $r>1/p$ is natural and assures the embedding in the space of continuous functions. It looks like the upper bound $2m-1+1/p$ is also sharp because $\boldsymbol{s}_{2m,\boldsymbol{j},\boldsymbol{k}} \not\in S_{p,\theta}^rB$ for $r\geq 2m-1+1/p$ and $\theta < \infty$.

Now we make some comments about F-case. Note that by using interpolation technique similar to Theorem 4.16 \cite{Glen2018} for the $F$-case the range of the smoothness parameter $r$ can be extended to $\max\{1/p,1/\theta\}<r<\min\{2m-1+\min\{1/p,1/\theta\},2m\}$. In the recent preprint \cite{Sr2020} it was proved that Chui-Wang wavelets constitute an unconditional basis in $S_{p,\theta}^rB$ only if $\max\{1/p,1/\theta\}-m<r<\min\{1/p,1/\theta\}+m-1$ (so far for univariate case only). It is an extension of the corresponding results for Haar wavelets \cite{SeU2015} for wavelets of higher regularity. We can say that our new basis functions  $\boldsymbol{s}_{2m,\boldsymbol{j},\boldsymbol{k}}$ are in some sense $m$-lifted Chui-Wang wavelets. Therefore the following restriction $\max\{1/p,1/\theta\}<r<2m-1+\min\{1/p,1/\theta\}$ for smoothness parameter $r$ is probably  also sharp.

Below on Figure \ref{smooth} we show the range of smoothness parameter $r$ for B-case as it is written in Theorem \ref{charact-m} and for F-case as it is written in this remark.
\end{rem}

\begin{rem}\label{rem-ding}
 In his papers \cite{DD2009}-\cite{DD2016} D. D\~{u}ng offers a different approach to sampling characterization of Besov spaces of mixed smoothness with higher regularity. The main difference with respect to our construction is that D. D\~{u}ng  considers ``frame-type''\ system while the system $\boldsymbol{s}_{2m;\jb,
 	\kb}$  is the ``basis-type''\ system in the sense that it is linearly independent.  Further we also prove unconditional convergence of this system in Besov-Triebel-Lizorkin spaces.
\end{rem}

\begin{figure}[h!]
	\begin{minipage}[h]{0.45\linewidth}
		\center{\includegraphics[width=1\linewidth]{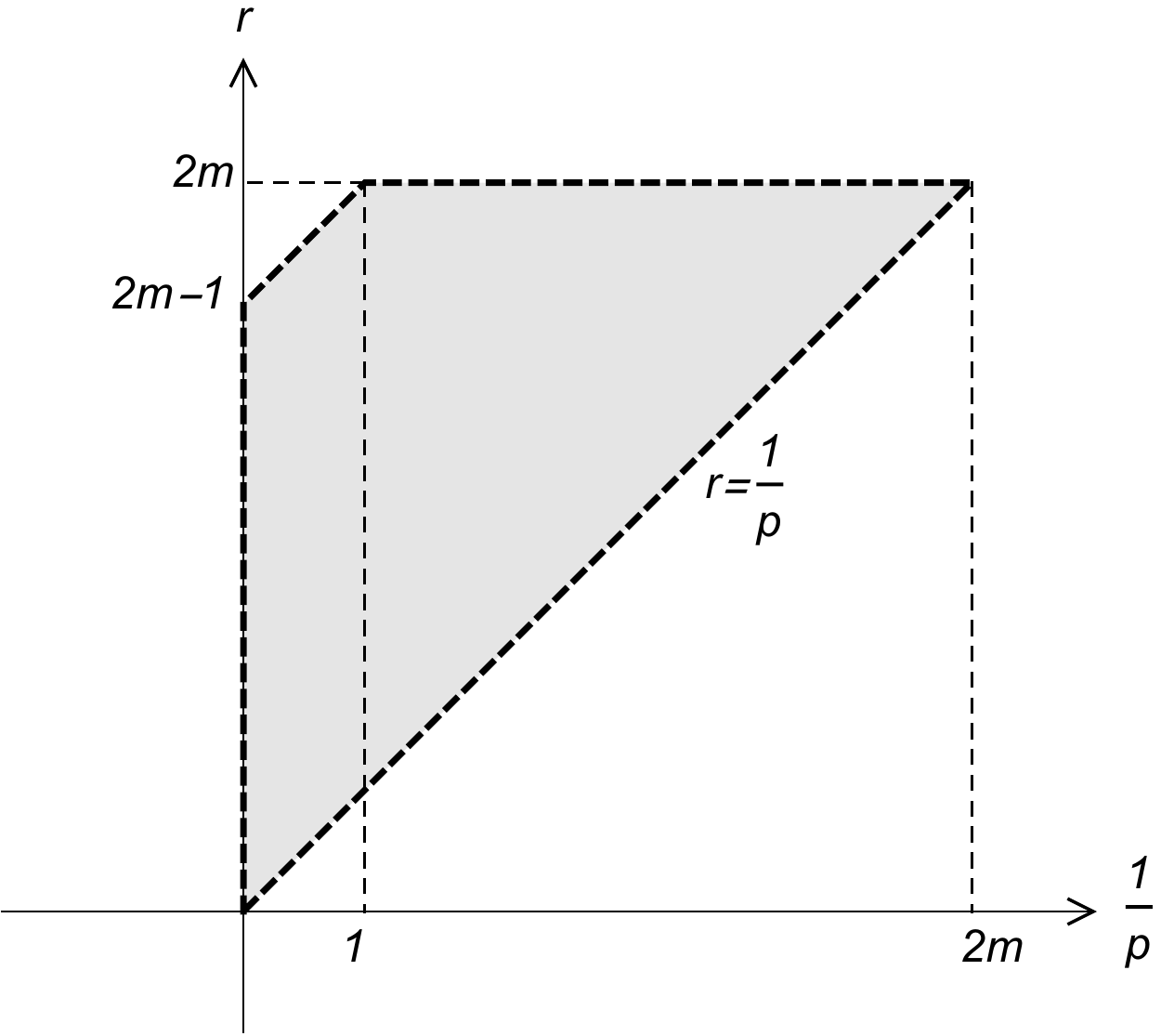}}
	\end{minipage}
	\hfill
	\begin{minipage}[h]{0.45\linewidth}
		\center{\includegraphics[width=1\linewidth]{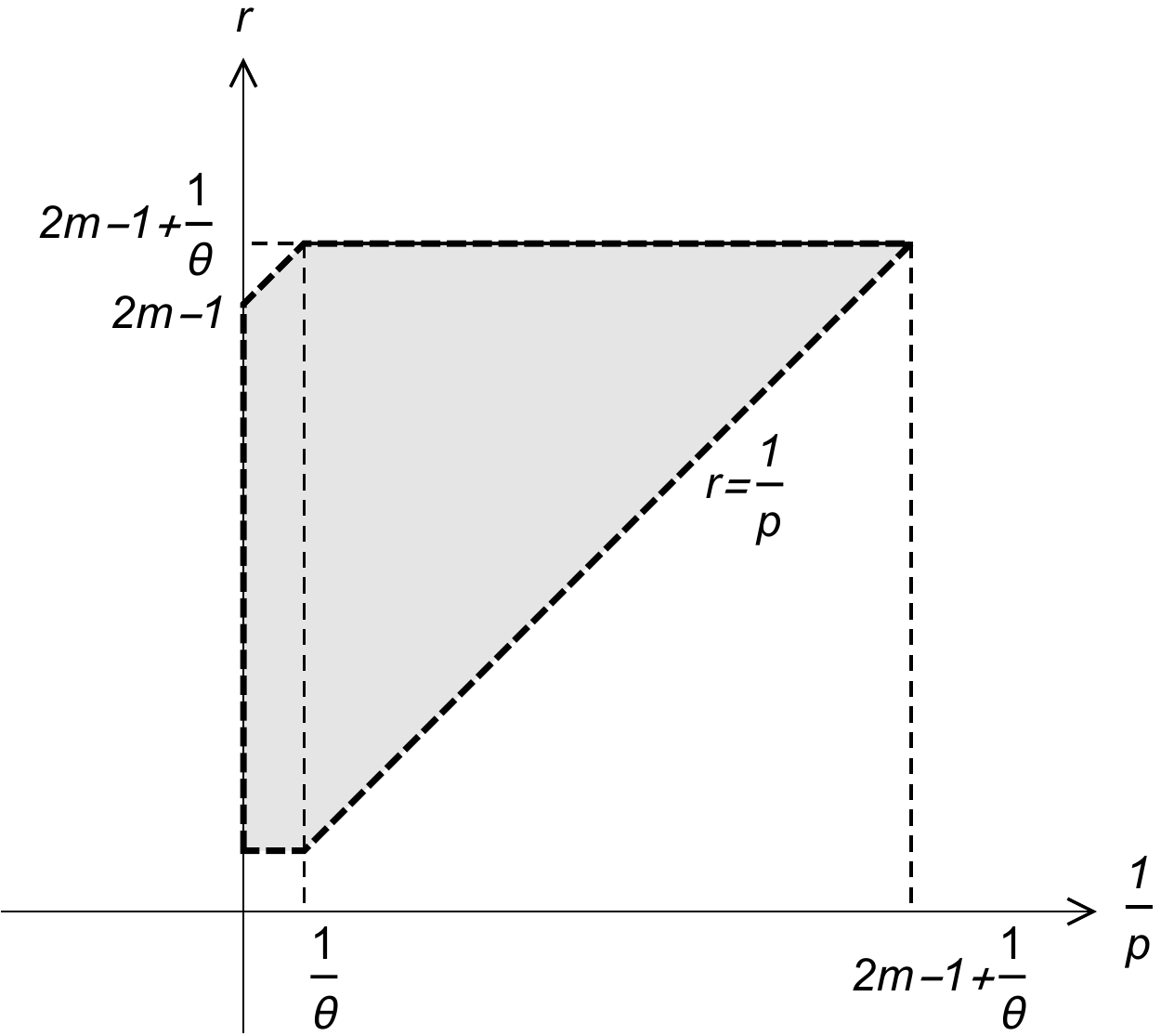}}
	\end{minipage}
	\caption{ Range of smoothness parameter $r$ for B-case (left) and for F-case (right, here $\theta>1$).}
	\label{smooth}
\end{figure}

First we prove some auxiliary statements. We apply  some known technique that was also used in \cite{Glen2018}, \cite{HMOT2016}, \cite{TV2019} and \cite{DU2019}.

The following lemma is a multivariate analog of Lemma A.11 from \cite{DU2019}. We formulate it without a proof.
\begin{lem} \label{convol}
	Let $\boldsymbol{j} \in \N_0^d$, $\boldsymbol{k}, \boldsymbol{l} \in \Z$ and $\boldsymbol{j}+\boldsymbol{l} \geq -1$. Then for the local means $\Psi_{\boldsymbol{j}}$ with  finitely many  vanishing moments of order $L$ the following estimates hold
	\begin{equation}\label{conv-b-1}
	|\Psi_{\boldsymbol{j}} * \boldsymbol{s}_{2m;\boldsymbol{j}+\boldsymbol{l},\kb}(\boldsymbol{x})|\leq C \prod\limits_{i=1}^d 2^{-\alpha_i|l_i|} \sum \limits_{\boldsymbol{n} \in \Z} |a_{\boldsymbol{n}}^{(m,e(\jb))}| \chi_{A_{\boldsymbol{j}+\boldsymbol{l},\boldsymbol{k}+\boldsymbol{n}}}(\boldsymbol{x})
	\end{equation}
	and
	\begin{equation}\label{conv-b-2}
	|\Psi_{\boldsymbol{j}} * \boldsymbol{s}_{2m;\boldsymbol{j}+\boldsymbol{l},\boldsymbol{k}}(\boldsymbol{x})|\leq C_R \prod\limits_{i=1}^d 2^{-\alpha_i|l_i|} (1+2^{\min\{j_i,j_i+l_i\}}|x_i-x_{j_i+l_i,k_i}|)^{-R},
	\end{equation}
	where $\alpha_i=1$ if $l_i\geq 0$, $\alpha_i=2m-1$ if $l_i<0$. The set $A_{\boldsymbol{j}+\boldsymbol{l},\boldsymbol{n}+\boldsymbol{k}}$ is the cross product of sets $A_{j+l,n+k}$ such that $A_{j+l,n+k}\subset \cup_{|u-k|\lesssim 2^{l_+}} I_{j+l_+,u+n}$.
\end{lem}
\begin{prop}\label{charact2} \begin{itemize}
		Let $m\in \N$ and $m\geq 2$.
		\item [(i)] 	For $0< p,\theta\leq \infty$, $p>1/2m$, $1/p<r<2m$ and  $f \in S_{p,\theta}^rB$  the inequality
		\begin{equation}\label{ch2}
		\|\lambda^{(m)}(f) \|_{s_{p,\theta}^rb}\lesssim \|f\|_{S_{p,\theta}^rB}
		\end{equation}
		holds.
		\item [(ii)] 	For $1/2m< p,\theta\leq \infty$, $p\neq \infty$, $\max\{1/p,1/\theta\}<r<2m$ and  $f \in S_{p,\theta}^rF$ the inequality
		\begin{equation}\label{chf2}
		\|\lambda^{(m)}(f) \|_{s_{p,\theta}^rf}\lesssim \|f\|_{S_{p,\theta}^rF}
		\end{equation}
		holds.
	\end{itemize}
\end{prop}
\begin{proof}
	We use the following pointwise representation of $f$ (formula (\ref{frep}))
	\begin{equation}\label{delta-char}
	f=\sum \limits_{\lb \in \Z} \delta_{\jb+\lb}[f], \ \ \jb \in \N_0^d.
	\end{equation}
Let us give a proof for the case $\theta<\infty$. For $\theta=\infty$ one can obtain the results by using similar technique with trivial modification. First we prove an additional inequality. We denote $F_{\jb,\lb}(\x):=\sum\limits_{\kb \in \Z} \lambda_{2m;\jb,\kb}(f)\left( \delta_{\jb+\lb}[f] \right) \chi_{\jb,\kb}(\x)$, $\x \in \R$.
	For $\x \in I_{\jb,\kb}$ we  have that
	\begin{equation}\label{f-oper}
	|F_{\jb,\lb}(\x)| \leq |\lambda_{2m;\jb,\kb}(f)\left( \delta_{\jb+\lb}[f] \right)| \leq \sum\limits_{\boldsymbol{s}\in \iota_m} \boldsymbol{N}_{2m}(\boldsymbol{s}+\boldsymbol{1}) |\Delta^{2m,e(\boldsymbol{j})}_{2^{-\boldsymbol{j}-1}}\delta_{\jb+\lb}[f](\boldsymbol{x}_{\boldsymbol{j};\boldsymbol{k},\boldsymbol{s}})|.
	\end{equation}
	
	Now we estimate difference $\Delta^{2m,e(\boldsymbol{j})}_{2^{-\boldsymbol{j}-1}}f(\boldsymbol{x}_{\boldsymbol{j};\boldsymbol{k},\boldsymbol{s}})$. We fixed direction $i$. If $i \in e(\boldsymbol{j})$ we do the following.
	By using Lemma \ref{peetre-ineq} we get for some bandlimited function $g$ with $\mathcal{F}g\subset [-A2^{j+l},B2^{j+l}]$
	$$
	|\Delta^{2m}_{2^{-j-1}} g(x_{j;k,s})|\lesssim \min\{1,2^{2ml}\} \max\{1,2^{al}\} P_{2^{l+j},a}g (x_{j;k,s}).
	$$
	From this inequality for $l<0$ and $|x-x_{j;k,s}|\leq C(m) 2^{-j}$ we get
	\begin{align}
	|\Delta^{2m}_{2^{-j-1}} g(x_{j;k,s})|& \lesssim 2^{2ml} P_{2^{l+j},a}g (x_{j;k,s})  \leq 2^{2ml} \sup\limits_{y \in \re} \frac{|g(y)|}{(1+2^{l+j}|y-x|)^a} (1+2^{l+j}|x-x_{j;k,s}|)^a\notag \\
	& \lesssim 2^{2ml} \sup\limits_{y \in \re} \frac{|g(y)|}{(1+2^{l+j}|y-x|)^a}=2^{2ml} P_{2^{l+j},a}g (x). \label{dif1}
	\end{align}
	
	For $l\geq 0$ and $|x-x_{j;k,s}|\leq C(m) 2^{-j}$ by using definition of the 2m-th order  difference we write
\begin{align*}
|\Delta^{2m}_{2^{-j-1}} g(x_{j;k,s})|& \lesssim \sup\limits_{|y|\lesssim 2^{-j}} |g(x_{j;k,s}+y)| \lesssim \sup\limits_{|y|\lesssim 2^{-j}} \frac{|g(x_{j;k,s}+y)|}{(1+2^j|y|)^a} \leq P_{2^j,a}(x_{j;k,s}) \\
&= \sup\limits_{y \in \re} \frac{|g(y)|}{(1+2^j|y-x|)^a}\frac{(1+2^j|y-x|)^a}{(1+2^j|y-x_{j;k,s}|)^a}  \\
&\leq \sup\limits_{y \in \re} \frac{|g(y)|}{(1+2^j|y-x|)^a} \frac{(1+2^j|y-x_{j;k,s}|)^a(1+2^j|x-x_{j;k,s}|)^a}{(1+2^j|y-x_{j;k,s}|)^a} \lesssim P_{2^j,a}(x) .
\end{align*}
Since for $l\geq 0$ we have
\begin{equation}\label{pin}
P_{2^j,a}(x)=\sup\limits_{y \in \re} \frac{|g(y)|}{(1+2^{j+l}|y-x|)^a} \frac{(1+2^{j+l}|y-x|)^a}{(1+2^{j}|y-x|)^a}  \leq 2^{la}P_{2^{l+j},a}g (x),
\end{equation}
then
\begin{equation}\label{dif2}
|\Delta^{2m}_{2^{-j-1}} g(x_{j;k,s})| \lesssim 2^{la}P_{2^{l+j},a}g (x).
\end{equation}

Let now $i \notin e(\boldsymbol{j})$. That means that $j=-1$. Let again $g$ be some bandlimited function. We consider only the case $l\geq 0$ because $g\equiv 0$ for $l<0$ and there is nothing to prove. So, for $l\geq 0$
\begin{align*}
|g(k)| \leq \sup \limits_{|y|\leq 1} |g(x+y)| =	\sup \limits_{|y|\leq 1} \dfrac{|g(x+y)|}{(1+2^j|y|)^a} (1+2^j|y|)^a \lesssim P_{2^j,a}(x).
\end{align*}	
Using (\ref{pin}) we get
\begin{equation}\label{dif3}
|g(k)| \lesssim 2^{la}P_{2^{l+j},a}g (x).
\end{equation}

Applying the pointwise inequalities (\ref{dif1}), (\ref{dif2}) and (\ref{dif3})  to the right side of (\ref{f-oper})
	\begin{align}\label{f-oper-1}
|F_{\jb,\lb}(\x)| & \lesssim \sum\limits_{\boldsymbol{s}\in \iota_m} \boldsymbol{N}_{2m}(\boldsymbol{s}+\boldsymbol{1}) P_{2^{\jb+\lb},a} \delta_{\jb+\lb}[f] \prod \limits_{i \in e(\jb)} \min\{2^{2ml_i},1\} \max\{2^{al_i},1\} \notag\\
& \lesssim  P_{2^{\jb+\lb},a} \delta_{\jb+\lb}[f] \prod \limits_{i \in e(\jb)} \min\{2^{2ml_i},1\} \max\{2^{al_i},1\}
\end{align}

	Let us now prove part (i). From definition of the space of sequences $b_{p,\theta}^r$ by using the $u$-triangle inequality with $u=\min\{p,\theta,1\}$ we can write
	\begin{align}
	\|\lambda^{(m)}(f)\|_{s_{p,\theta}^rb}& = 	\Bigg(\sum\limits_{\jb \in \N^d_{-1}} 2^{\theta r |\jb|_1} \Big\|\sum\limits_{\kb \in \Z} \lambda_{2m;\jb,\kb}(f)(f) \chi_{\jb,\kb} \Big\|_p^\theta \Bigg)^{1/\theta} \notag\\
	&=	\Bigg(\sum\limits_{\jb \in \N^d_{-1}} 2^{\theta r |\jb|_1} \Big\|\sum\limits_{\kb \in \Z} \lambda_{2m;\jb,\kb}(f)\Big(\sum \limits_{\lb \in \Z} \delta_{\jb+\lb}[f] \Big) \chi_{\jb,\kb} \Big\|_p^\theta \Bigg)^{1/\theta} \notag\\
	& \leq \Bigg( \sum\limits_{\lb \in \Z} \Big( \sum\limits_{\jb\in \N^d_{-1}} 2^{\theta r |\jb|_1} \Big\|\sum\limits_{\kb \in \Z} \lambda_{2m;\jb,\kb}(f)\left( \delta_{\jb+\lb}[f] \right) \chi_{\jb,\kb} \Big\|_p^\theta \Big)^{u/\theta}\Bigg)^{1/u}. \label{lamb-norm}
	\end{align}
	
	By using Lemma \ref{peetre-ineq1} and inequality (\ref{f-oper-1}) we have that for $a>1/p$
	\begin{align*}
	\|F_{\jb,\lb}\|_p & \lesssim \prod \limits_{i\in e(\jb)}\min\{2^{2ml_i},1\} \max\{2^{al_i},1\} \, \|P_{2^{\jb+\lb},a}\delta_{\jb+\lb}[f]\|_p \\
	& \lesssim \prod \limits_{i\in e(\jb)} \min\{2^{2ml_i},1\} \max\{2^{al_i},1\} \, \|\delta_{\jb+\lb}[f]\|_p.
	\end{align*}
	
	We denote $\beta_i=a$ if $l_i\geq 0$ and $\beta_i=2m$ if $l_i<0$. Now taking into account Definition \ref{local-mean} we can proceed estimation (\ref{lamb-norm})
	\begin{align*}
	\|\lambda^{(m)}(f)\|_{s_{p,\theta}^rb}
	&\leq \Bigg(\sum\limits_{\lb \in \Z} \Big( \sum\limits_{\jb\in \N^d_{-1}} 2^{\theta r |\jb|_1} \|F_{\jb,\lb}\|_p^\theta \Big)^{u/\theta}\Bigg)^{1/u}\\
	&\lesssim \Bigg(\sum\limits_{\lb \in \Z} \Big(\prod\limits_{i=1}^d 2^{l_i(\beta_i-r)} \Big)^u \Big( \sum\limits_{\jb\in \N^d_{-1}} 2^{\theta r (|\jb|_1+|\lb|_1)} \|\delta_{\jb+\lb}[f]\|_p^\theta \Big)^{u/\theta}\Bigg)^{1/u}\\
	&\leq\Big(\sum\limits_{\lb \in \Z} \prod\limits_{i=1}^d 2^{l_i(\beta_i-r)u} \Big)^{1/u}\|f\|_{S_{p,\theta}^rB}.
	\end{align*}
	Now if $r$ satisfies $1/p<a<r<2m$ with $p>1/2m$ we have that
	for $l_i\in \zz$ holds $2^{(\beta_i-r)l_i}\leq 2^{-\gamma_i |l_i|}$ where $\gamma_i=r-a$ if $l_i\geq 0$ and $\gamma_i=2m-r$ if $l_i<0$. It implies that the series $\sum_{\lb \in \Z} 2^{- u \boldsymbol{\gamma} \cdot |\lb|}$
	 converges and the inequality (\ref{ch2}) holds.

	Now we prove part (ii). We use the representation (\ref{delta-char}) and the $u$-triangle inequality
	\begin{align*}
	\|\lambda^{(m)}(f)\|_{s_{p,\theta}^{r}f}&=
	\Big\| \Big(\sum\limits_{\jb \in \N_{-1}^d} 2^{\theta r |\jb|_1} \Big|\sum\limits_{\kb \in \Z} \lambda_{2m;\jb,\kb}(f)(f) \chi_{\jb,\kb} \Big|^\theta \Big)^{1/\theta}\Big\|_p\\
	&=\Big\| \Big(\sum\limits_{\jb \in \N_{-1}^d} 2^{\theta r |\jb|_1} \Big|\sum\limits_{\kb \in \Z} \lambda_{2m;\jb,\kb}(f)\Big(\sum \limits_{\lb \in \Z} \delta_{\jb+\lb}[f]\Big) \chi_{\jb,\kb} \Big|^\theta \Big)^{1/\theta}\Big\|_p\\
	& \leq \Bigg(\sum \limits_{\lb \in \Z} \Big\| \Big(\sum\limits_{\jb \in \N_{-1}^d} 2^{\theta r |\jb|_1} \Big|\sum\limits_{\kb \in \Z} \lambda_{2m;\jb,\kb}(f)\Big( \delta_{\jb+\lb}[f]\Big) \chi_{\jb,\kb} \Big|^\theta \Big)^{1/\theta}\Big\|_p^u \Bigg)^{1/u}.
	\end{align*}
	By using inequality (\ref{f-oper-1}) and Lemma \ref{peetre-ineq2} we write for $a>\max\{1/p,1/\theta\}$
	\begin{align*}
	\|\lambda^{(m)}(f)\|_{s_{p,\theta}^{r}f}& \lesssim \Bigg(\sum \limits_{\lb \in \Z} \Big(\prod \limits_{i=1}^d2^{\beta_i l_i } \Big)^u \Big\| \Big(\sum\limits_{\jb \in \N_{-1}^d} 2^{\theta r |\jb|_1} \big(P_{2^{\jb+\lb},a}\delta_{\jb+\lb}[f]\big)^\theta \Big)^{1/\theta}\Big\|_p^u\Bigg)^{1/u} \\
	& =\Bigg(\sum \limits_{\lb \in \Z} \Big(\prod \limits_{i=1}^d2^{(\beta_i-r) l_i } \Big)^u\Big\| \Big(\sum\limits_{\jb \in \N_{-1}^d} 2^{\theta r (|\jb|_1+|\lb|_1)} \big(\delta_{\jb+\lb}[f]\big)^\theta \Big)^{1/\theta}\Big\|_p^u\Bigg)^{1/u}\\
	&\leq\Big(\sum\limits_{\lb \in \Z} \prod\limits_{i=1}^d 2^{l_i(\beta_i-r)u} \Big)^{1/u}\|f\|_{S_{p,\theta}^rF},
	\end{align*}
	Due to the choice of the parameter $\max\{1/p,1/\theta\}<a<r<2m$ with $p,\theta>1/2m$ the series in the last inequality is convergent (see explanation above) and inequality (\ref{chf2}) holds.
\end{proof}

 Let $\lambda^{(m)}:=\{\lambda_{2m;\jb,\kb}: \, \boldsymbol{j}\in \N_{-1}^d$, $\boldsymbol{k} \in \Z \}$ be some sequence of real numbers that satisfy certain conditions (later we specify these conditions). We denote
\begin{equation}\label{series}
f:=\sum \limits_{\jb \in \N^d_{-1}} \sum \limits_{\kb \in \Z} \lambda_{2m;\jb,\kb} \boldsymbol{s}_{2m;\jb,\kb}.
\end{equation}
\begin{prop}\label{charact1}
Let $0< p,\theta\leq \infty$, $m \in \N$ and $m\geq 2$, $\max\{1/p-1,0\}<r<2m-1+1/p$ and a sequence $\lambda^{(m)} \in s_{p,\theta}^rb$. Then the series (\ref{series}) converges unconditionally in the space $S_{p,\theta}^{r-\varepsilon}B$  for every $\varepsilon>0$. If $\max\{p,\theta\}<\infty$ we have unconditional convergence in the space $S_{p,\theta}^{r}B$. Moreover, the following inequality holds
\begin{equation}\label{ch1}
\|f\|_{S_{p,\theta}^rB}  \lesssim \|\lambda^{(m)} \|_{s_{p,\theta}^rb}.
\end{equation}
\end{prop}
\begin{proof}	
First we prove the inequality (\ref{ch1}) for the case $\theta<\infty$. For $\theta=\infty$ the proof is similar.
We denote $f_{\jb}:=\sum \limits_{\kb \in \Z} \lambda_{2m;\jb,\kb} \boldsymbol{s}_{2m;\jb,\kb}$ for $\jb \in \N^d_{-1}$. Then
\begin{equation}\label{repr}
f=\sum \limits_{\lb \in \Z} f_{\jb+\lb}.
\end{equation}
By using characterization of Besov spaces via local means (Theorem \ref{local-mean})
and $u$-triangle inequality with $u:=\min\{p,\theta,1\}$ we have
\begin{align*}
\|f\|_{S_{p,\theta}^rB}&\asymp\Big( \sum\limits_{\jb \in \N^d_0} 2^{\theta r |\jb|_1 } \|\Psi_{\jb} * f\|_p^\theta \Big)^{1/\theta}  \\
&=\Big( \sum\limits_{\jb \in \N^d_0} 2^{\theta r |\jb|_1 } \Big\|\Psi_{\jb} * \Big(\sum \limits_{\lb \in \Z} \sum \limits_{\kb \in \Z} \lambda_{2m;\jb+\lb,\kb} \boldsymbol{s}_{\jb+\lb,\kb} \Big)\Big\|_p^\theta \Big)^{1/\theta} \\
& \leq \Big( \sum \limits_{\lb \in \Z} \Big( \sum\limits_{\jb \in \N^d_0} 2^{\theta r |\jb|_1} \Big\| \sum \limits_{\kb \in \Z} \lambda_{2m;\jb+\lb,\kb} (\Psi_{\jb}*\boldsymbol{s}_{\jb+\lb,\kb}) \Big\|_p^\theta   \Big)^{u/\theta}\Big)^{1/u}.
\end{align*}
By using inequality (\ref{conv-b-1}) we can proceed for $v=\min\{p,1\}$
\begin{align}
\|f\|_{S_{p,\theta}^rB}& \lesssim \Bigg(\sum \limits_{\lb \in \Z} \Bigg( \sum\limits_{\jb \in \N^d_0} 2^{\theta r \jb} \Big\| \sum \limits_{\kb \in \Z} \lambda_{2m;\jb+\lb,\kb} \, 2^{-\boldsymbol{\alpha} \cdot|\lb|} \sum \limits_{\nb \in \Z} |a_{\nb}| \chi_{A_{\jb+\lb,\kb+\nb}}(\cdot) \Big\|_p^\theta   \Bigg)^{u/\theta} \Bigg)^{1/u}\notag \\
& \leq \Bigg( \sum \limits_{\lb \in \Z} \Bigg( \sum\limits_{\jb \in \N^d_0} 2^{\theta r |\jb|_1} 2^{-\boldsymbol{\alpha} \cdot|\lb| \theta} \Big( \sum \limits_{\nb \in \Z} |a_{\nb}|^v \Big\| \sum \limits_{\kb \in \Z} \lambda_{2m;\jb+\lb,\kb} \, \chi_{A_{\jb+\lb,\kb+\nb}}(\cdot) \Big\|^v_p\Big)^{\theta/v}  \Bigg)^{u/\theta}\Bigg)^{1/u} .  \label{norm}
\end{align}
Further we consider the following norm $\big\| \sum_{\kb \in \Z} \lambda_{2m;\jb+\lb,\kb} \,  \chi_{A_{\jb+\lb,\kb+\nb}}(\cdot) \big\|_p$. For $\x \in \R$ since $A_{\jb+\lb,\nb+\kb}\subset \bigcup_{|\tb-\kb|\lesssim 2^{\lb_+}} I_{\jb+\lb_+,\tb+\nb}$ we can write
\begin{align*}
\Big|\sum \limits_{\kb \in \Z} \lambda_{2m;\jb+\lb,\kb} \,  \chi_{A_{\jb+\lb,\kb+\nb}}(\x) \Big|^p &\leq \Big|\sum \limits_{\kb \in \Z} |\lambda_{2m;\jb+\lb,\kb}| \,  \chi_{A_{\jb+\lb,\kb+\nb}}(\x) \Big|^p\\
&\leq \Big|\sum \limits_{\kb \in \Z} |\lambda_{2m;\jb+\lb,\kb}| \sum \limits_{\tb \in G_{\lb}(\kb)} \chi_{I_{\jb+\lb_+,\tb+\nb}}(\x) \Big|^p,
\end{align*}
where $G_{\lb}(\kb):=\{\tb: \, |\tb-\kb|\lesssim 2^{\lb_+}\}$ with $|G_{\lb}(\kb)|\asymp 2^{\lb+}$. By changing order of summation and on the viewpoint that segments $I_{\jb+\lb_+,\tb+\nb}$ do not intersect for different $\tb $ we have
\begin{align*}
\Big|\sum \limits_{\kb \in \Z} \lambda_{2m;\jb+\lb,\kb} \, \chi_{A_{\jb+\lb,\kb+\nb}}(\x) \Big|^p &\leq \Big|\sum \limits_{\tb \in \Z} \chi_{I_{\jb+\lb_+,\tb+\nb}}(\x) \sum \limits_{\kb \in G_{\lb}(\tb)} |\lambda_{2m;\jb+\lb,\kb}|  \Big|^p\\
&=\sum \limits_{\tb \in \Z} \chi_{I_{\jb+\lb_+,\tb+\nb}}(\x) \Big( \sum \limits_{\kb \in G_{\lb}(\tb)} |\lambda_{2m;\jb+\lb,\kb}| \Big)^p.
\end{align*}
From the H\"{o}lder inequality for $p>1$ we obtain
$$
\Big|\sum \limits_{\kb \in \Z} \lambda_{2m;\jb+\lb,\kb} \,  \chi_{A_{\jb+\lb,\kb+\nb}}(\x) \Big|^p \lesssim   2^{|\lb_+|_1(p-1)} \sum \limits_{\tb \in \Z} \chi_{I_{\jb+\lb_+,\tb+\nb}}(\x) \sum \limits_{\kb \in G_{\lb}(\tb)} |\lambda_{2m;\jb+\lb,\kb}|^p.
$$
For $p<1$ we use the embedding $l_p \hookrightarrow l_1$ to get
$$
\Big|\sum \limits_{\kb \in \Z} \lambda_{2m;\jb+\lb,\kb} \, \chi_{A_{\jb+\lb,\kb+\nb}}(\x) \Big|^p \leq  \sum \limits_{\tb \in \Z} \chi_{I_{\jb+\lb_+,\tb+\nb}}(\x) \sum \limits_{\kb \in G_{\lb}(\tb)} |\lambda_{2m;\jb+\lb,\kb}|^p.
$$
By using last inequality we can write for the norm
\begin{align*}
\Big\|\sum \limits_{\kb \in \Z} \lambda_{2m;\jb+\lb,\kb} \,  \chi_{A_{\jb+\lb,\kb+\nb}} \Big\|_p^p &= \int\limits_{\R} \Big|\sum \limits_{\kb \in \Z} \lambda_{2m;\jb+\lb,\kb} \, \chi_{A_{\jb+\lb,\kb+\nb}}(\x) \Big|^p \, d\x \\
& \lesssim  2^{|\lb_+|_1(p-1)_+}  \sum \limits_{\tb \in \Z} \int\limits_{\R} \chi_{I_{\jb+\lb_+,\tb+\nb}} (\x) \, d\x \sum \limits_{\kb \in G_{\lb}(\tb)} |\lambda_{2m;\jb+\lb,\kb}|^p\\
&=2^{|\lb_+|_1(p-1)_+} 2^{-|\jb|_1-|\lb_+|_1} \sum \limits_{\tb \in \Z}\sum \limits_{\kb\in G_{\lb}(\tb)} |\lambda_{2m;\jb+\lb,\kb}|^p\\
& \asymp 2^{|\lb_+|_1(p-1)_+}2^{-|\jb|_1} \sum \limits_{\kb \in \Z}|\lambda_{2m;\jb+\lb,\kb}|^p.
\end{align*}

 By using this inequality we can continue the estimation of (\ref{norm})
\begin{align*}
\|f\|_{S_{p,\theta}^rB}& \lesssim   \Bigg(  \sum \limits_{\lb \in \Z} \Bigg( \sum\limits_{\jb \in \N^d_0} 2^{\theta r |\jb|_1-\boldsymbol{\alpha} \cdot |\lb| \theta }  \Big( \sum \limits_{\nb \in \Z} |a_{\nb}|^v 2^{v(|\lb_+|_1(1-1/p)_+-|\jb|_1/p)} \Big(\sum \limits_{\kb \in \Z}|\lambda_{2m;\jb+\lb,k}|^p \Big)^{\frac{v}{p}} \Big)^{\frac{\theta}{v}}  \Bigg)^{\frac{u}{\theta}} \Bigg)^{\frac{1}{u}} \\
& = \Big(\sum \limits_{\nb \in \Z} |a_{\nb}|^v \Big)^{\frac{1}{v}} \cdot  \Bigg(  \sum \limits_{\lb \in \Z} \Bigg( \sum\limits_{\jb \in \N^d_0} 2^{\theta (r |\jb|_1-\boldsymbol{\alpha} \cdot|\lb| +|\lb_+|_1(1-1/p)_+-|\jb|_1/p)}  \Big(\sum \limits_{\kb \in \Z}|\lambda_{2m;\jb+\lb,\kb}|^p \Big)^{\frac{\theta}{p}} \Bigg)^{\frac{u}{\theta}}\Bigg)^{\frac{1}{u}}.
\end{align*}
From definition of coefficients $a_{\nb}$ we conclude that $\sum_{\nb \in \Z} |a_{\nb}|^v < \infty$, so we can proceed as follows
 \begin{align*}
 \|f\|_{S_{p,\theta}^rB}& \lesssim   \Bigg( \sum \limits_{\lb \in \Z} 2^{-\boldsymbol{\alpha} \cdot|\lb|u} 2^{|\lb_+|_1(1-1/p)_+u} \Bigg( \sum\limits_{\jb \in \N^d_0} 2^{\theta |\jb|_1 (r-1/p)}   \Big(\sum \limits_{\kb \in \Z}|\lambda_{2m;\jb+\lb,\kb}|^p \Big)^{\frac{\theta}{p}} \Bigg)^{\frac{u}{\theta}}\Bigg)^{\frac{1}{u}}\\
 &=\Bigg(\sum \limits_{\lb \in \Z} 2^{-\boldsymbol{\alpha}\cdot|\lb|} 2^{|\lb_+|_1(1-1/p)_+} 2^{-|\lb|(r-1/p)} \Big( \sum\limits_{\jb \in \N^d_0} 2^{\theta (|\jb|_1+|\lb|_1) (r-1/p)}   \Big(\sum \limits_{\kb \in \Z}|\lambda_{2m;\jb+\lb,\kb}|^p \Big)^{\frac{\theta}{p}} \Big)^{\frac{u}{\theta}}\Bigg)^{\frac{1}{u}}\\
 & \leq \Big( \sum \limits_{\lb \in \Z} 2^{-\boldsymbol{\alpha}\cdot|\lb|u} 2^{|\lb_+|_1(1-1/p)_+u} 2^{-|\lb|_1(r-1/p)u} \Big)^{\frac{1}{u}} \|\lambda^{(m)}\|_{s_{p,\theta}^rb}.
 \end{align*}
Due to the choice of the parameter $\max\{1/p-1,0\}<r<2m-1+1/p$ the series in the last inequality
 is convergent. Therefore, inequality (\ref{ch1}) holds.

The unconditional convergence of the series (\ref{series}) in the space $S_{p,\theta}^rB$ can be proved by using similar technique as in \cite[Proposition 4.2]{DU2019}
\end{proof}

Now we prove the analogue of this proposition for $F$-spaces.

\begin{prop}\label{charactf1}
	Let $0< p,\theta\leq \infty$, $p\neq \infty$, $m\in \N$ and $m\geq 2$, $\max\{1/\theta-1,1/p-1,0\}<r<2m-1$ and a sequence $\lambda^{(m)} \in s_{p,\theta}^rf$. Then the series (\ref{series})
	converges unconditionally in the space $S_{p,\theta}^{r-\varepsilon}F$  for every $\varepsilon>0$. If $\theta<\infty$ we have unconditional convergence in the space $S_{p,\theta}^rF$. Moreover, the following inequality holds
	$$
	\|f\|_{S_{p,\theta}^rF} \lesssim \|\lambda^{(m)} \|_{s_{p,\theta}^rf}.
$$
\end{prop}
\begin{proof}
We use $u$-triangle inequality with $u=\min\{1,p,q\}$, representation (\ref{repr}) and  Theorem \ref{local-mean}
\begin{align*}
\|f\|_{S_{p,\theta}^{r}F} &\asymp
\Big\|\Big(\sum\limits_{\jb \in \N^d_0} 2^{\theta r |\jb|_1} |\Psi_{\jb}*f|^\theta   \Big)^{1/\theta}\Big\|_p \\
&= \Big\|\Big(\sum\limits_{\jb \in \N^d_0} 2^{\theta r |\jb|_1} \Big|\Psi_{\jb}*\Big(\sum \limits_{\lb \in \Z} \sum \limits_{\kb \in \Z} \lambda_{2m;\jb+\lb,\kb} \, \boldsymbol{s}_{2m;\jb+\lb,\kb} \Big)\Big|^\theta   \Big)^{1/\theta}\Big\|_p \\
& \leq \Bigg( \sum \limits_{\lb \in \Z} \Big\|\Big(\sum\limits_{\jb \in \N^d_0} 2^{\theta r |\jb|_1} \Big|\Psi_{\jb}*\Big( \sum \limits_{\kb \in \Z} \lambda_{2m;\jb+\lb,\kb} \, \boldsymbol{s}_{2m;\jb+\lb,\kb} \Big)\Big|^\theta   \Big)^{1/\theta}\Big\|_p^u\Bigg)^{1/u}\\
& \leq \Bigg( \sum \limits_{\lb \in \Z} \Big\|\Big(\sum\limits_{\jb \in \N^d_0} 2^{\theta r |\jb|_1} \Big( \sum \limits_{\kb \in \Z} |\lambda_{2m;\jb+\lb,\kb}| \, |\Psi_{\jb}*\boldsymbol{s}_{2m;\jb+\lb,\kb}| \Big)^\theta   \Big)^{1/\theta}\Big\|_p^u\Bigg)^{1/u}.
\end{align*}
By using inequality (\ref{conv-b-2}) we obtain
$$
\|f\|_{S_{p,\theta}^rF} \lesssim \Bigg(  \sum \limits_{\lb \in \Z} 2^{-\boldsymbol{\alpha} \cdot |\lb|u} \Big\|\Big(\sum\limits_{\jb \in \N^d_0} 2^{\theta r |\jb|_1} \Big( \sum \limits_{\kb \in \Z} |\lambda_{2m;\jb+\lb,\kb}| \prod\limits_{i=1}^d(1+2^{\min\{j_i,j_i+l_i\}}|x_i-x_{j_i+l_i,k_i}|)^{-R} \Big)^\theta   \Big)^{\frac{1}{\theta}}\Big\|_p^u \Bigg)^{\frac{1}{u}}.
$$
From the following property (\cite[Lem. 7.1]{Kyr2003})
$$
\sum\limits_{\kb \in \Z} |\lambda_{2m;\jb+\lb,\kb}| \prod\limits_{i=1}^d (1+2^{\min\{j_i,j_i+l_i\}}|x_i-x_{j_i+l_i,k_i}|)^{-R}\lesssim 2^{|\lb_+|_1/\tau} \Big[M\Big|\sum\limits_{\kb \in \Z} \lambda_{2m;\jb+\lb,\kb} \chi_{\jb+\lb,\kb}\Big|^\tau \Big]^{1/\tau}(\x)
$$
for $0<\tau \leq 1$ and $R>1/\tau$, we get
$$
\|f\|_{S_{p,\theta}^{r}F} \lesssim \Bigg(  \sum \limits_{\lb \in \Z} 2^{-\boldsymbol{\alpha} \cdot|\lb|u} 2^{u |\lb_+|_1/\tau} \Big\|\Big(\sum\limits_{\jb \in \N^d_0} 2^{\theta r |\jb|_1}  \Big[M\Big|\sum\limits_{\kb \in \Z} \lambda_{2m;\jb+\lb,\kb} \, \chi_{\jb+\lb,\kb}\Big|^\tau \Big]^{\theta/\tau} \Big)^{1/\theta}\Big\|_p^u \Bigg)^{1/u}.
$$
It is obvious that
	\begin{equation}\label{m-ineq}
\Big\|\Big( \sum \limits_{\lb} \big[M|f_{\lb}|^\tau\big]^{\theta/\tau}\Big)^{1/\theta} \Big\|_p=\Big\|\Big( \sum \limits_{\lb} \big[M|f_{\lb}|^\tau\big]^{\theta/\tau}\Big)^{\tau/\theta} \Big\|_{p/\tau}^{1/\tau}.
	\end{equation}
We assume that $\min\{\theta/\tau,p/\tau\}>1$.
By using the Hardy-Littlewood maximal inequality we have
\begin{align*}
\|f\|_{S_{p,\theta}^{r}F} &\lesssim \Bigg( \sum \limits_{\lb \in \Z} 2^{-\boldsymbol{\alpha} \cdot|\lb|u} 2^{u |\lb_+|_1/\tau} \Big\|\Big(\sum\limits_{\jb \in \N^d_0} 2^{\theta r |\jb|_1}  \Big|\sum\limits_{\kb \in \Z} \lambda_{2m;\jb+\lb,\kb} \, \chi_{\jb+\lb,\kb}\Big|^{\theta} \Big)^{1/\theta}\Big\|_p^u \Bigg)^{1/u}\\
&=\Bigg( \sum \limits_{\lb \in \Z} 2^{-\boldsymbol{\alpha} \cdot|\lb|u} 2^{u |\lb_+|_1/\tau}2^{-ru |\lb|_1} \Big\|\Big(\sum\limits_{\jb \in \N^d_0} 2^{\theta r (|\jb|_1+|\lb|_1)}  \Big|\sum\limits_{\kb \in \Z} \lambda_{2m;\jb+\lb,\kb} \, \chi_{\jb+\lb,\kb}\Big|^{\theta} \Big)^{1/\theta}\Big\|_p^u\Bigg)^{1/u}\\
& \leq \Big(\sum \limits_{\lb \in \Z} 2^{-\boldsymbol{\alpha} \cdot|\lb|u} 2^{ u|\lb_+|_1/\tau}2^{-ru |\lb|_1}\Big)^{1/u} \|\lambda^{(m)}\|_{s_{p,\theta}^rf}
\end{align*}
for $\tau<\min\{1,p,\theta\}$.
If $\max\{1/\theta-1,1/p-1,0\}\leq 1/\tau-1<r<2m-1$ then the series $\sum_{l \in \zz} 2^{-\alpha |l|u} 2^{ ul_+/\tau}2^{-rlu}$ converges.
\end{proof}	

\textbf{Proof of Theorem \ref{charact-m}.} This theorem can be proved by a similar technique as in the proof of Theorem 4.1 \cite{DU2019}  with the use of Theorem \ref{conv-c} and Propositions \ref{charact2}, \ref{charact1} and \ref{charactf1}.
$\blacksquare$

\section{Best $n$-term approximation with respect to higher order Faber splines}

In this section we consider the quantity of best $n$-term approximation with respect to higher order Faber spline basis. In \cite[Chapter 6]{Glen2018} the author obtain the corresponding order estimates with respect to Faber-Schauder basis. These results hold for restricted smoothness $r$ due to restricted smoothness of Faber-Schauder basis functions.  Our goal was to  extend these results for Besov and Lizorkin-Triebel spaces with higher regularity.

First we give a definition of the quantity of best $n$-term approximation. Let $X$ be a Banach space and $U=\{u_\alpha\}_{\alpha \in \Omega}$ be a system of elements from $X$ such that $\overline{\Span U}=X$. Here $\Omega$ is a countable set of indices, in particular $\Omega=\Z$. The quantity of best $n$-term approximation of an element $f \in X$ with respect to the system $U$ is defined as
 $$
 \sigma_n(f,U,X):= \inf_{\substack{\Lambda(f)\subset \Omega \\ \#\Lambda(f) = n}} \, \inf \limits_{c_\alpha \in \re}  \Big\| f - \sum\limits_{\alpha\in \Lambda(f)} c_\alpha u_\alpha\Big\|_X.
 $$
For some subset $F\subset X$
	\begin{equation}\label{best-n}
\sigma_n(F,U,X):= \sup_{f\in F}  \sigma_n(f,U,X).
\end{equation}

 By $\nabla$ we define a set of indices for the basis  $\boldsymbol{s}_{2m;\jb,\kb}$, i.e
$$
\nabla=\{(\jb,\kb): \jb\in \N_{-1}^d, \kb \in \Z\},
$$
and we denote $\mathcal{B}_{2m}^d=\{\boldsymbol{s}_{2m;\jb,\kb}: \, (\jb,\kb)\in \nabla\}$.  Let us fix a compact set $K\subset \R$. We consider approximation of functions $f$ from Besov $S_{p,\theta}^rB(\R)$ and Lizorkin-Triebel $S_{p,\theta}^rF(\R)$ spaces in the space $L_q(K)$. Further for simplicity instead of the norm $\|\cdot\|_{L_q(K)}$ we will write $\|\cdot\|_{q}$ and instead of $\sigma_n(S_{p,\theta}^rB(\R),\mathcal{B}_{2m}^d, L_q(K))$ we write simply $\sigma_n(S_{p,\theta}^rB,\mathcal{B}_{2m}^d)_q$.

\subsection{Upper order estimates.}  In this subsection we present our main results with respect to upper order estimates of best $n$-term approximation of Besov and Lizorkin-Triebel spaces. Let us first give some notations. Let $\Omega$ be some compact set such that $K\subset \Omega \subset \R$. We define a function $\phi$ by the following three properties: 1) $\phi \in C^\infty$; 2) $\phi(\x)=1$ for $\x \in K$; 3) $\supp \phi=\Omega$.

 To get results we substantially use the estimates for best $n$-term approximation for sequence spaces $s_{p,\theta}^r b(K)$ and $s_{p,\theta}^r f(K)$ that were obtained by Hansen and Sickel  \cite{HS2012}.

\begin{satz} \label{best-m-term-1} Let $1/2m<p<q\leq \infty$, $\theta\leq \min\{q,1\}$ and $1/p<r<\min\{2m,1/\theta-1/\min\{q,1\}\}$ or $1/p<r=1/\theta-1/\min\{q,1\}<2m$. Then 	
$$
\sigma_n(S_{p,\theta}^rB,\mathcal{B}_{2m}^d)_q\lesssim n^{-r}.
$$
\end{satz}
\begin{proof}
Let $f \in S_{p,\theta}^rB(\R)$. Then $f \phi$ is compactly supported and according to 	Theorem \ref{charact-m} can be expanded in the series (\ref{exp-mul}). On the other hand, from Theorem 1.3 of the paper \cite{NUU2017} for $r>\max\{1/p-1,0\}$ we have $\|f\phi\|_{S_{p,\theta}^rB(\R)}\leq C(\phi)\|f\|_{S_{p,\theta}^rB(\R)} $. Therefore, we can write
\begin{align*}
\sigma_n(S_{p,\theta}^rB,\mathcal{B}_{2m}^d)_q&=\sup \limits_{f \in S_{p,\theta}^rB(\R)} \, \inf_{\substack{\Lambda\subset \nabla \\ \#\Lambda = n}} \, \inf \limits_{\mu_{\jb,\kb}}  \Big\| f - \sum\limits_{(\jb,\kb)\in \Lambda} \mu_{\jb,\kb} \boldsymbol{s}_{2m;\jb,\kb} \Big\|_q\\
&\lesssim \sup \limits_{(f\phi) \in S_{p,\theta}^rB(\R)} \, \inf_{\substack{\Lambda\subset \nabla \\ \#\Lambda = n}} \, \inf \limits_{\mu_{\jb,\kb}}  \Big\| f\phi - \sum\limits_{(\jb,\kb)\in \Lambda} \mu_{\jb,\kb} \boldsymbol{s}_{2m;\jb,\kb} \Big\|_q\\
&= \sup \limits_{(f\phi) \in S_{p,\theta}^rB(\R)} \, \inf_{\substack{\Lambda\subset \nabla \\ \#\Lambda = n}} \, \inf_{\substack{\mu_{\jb,\kb} \\ \mu_{\jb,\kb}=0 \, \text{if} \, (\jb,\kb)\not\in \Lambda}} \Big\| \sum \limits_{\jb \in \N_{-1}^d}\sum\limits_{\kb \in \Z} (\lambda_{2m;\jb,\kb}(f\phi) -\mu_{\jb,\kb}) \boldsymbol{s}_{2m;\jb,\kb}\Big\|_q.
\end{align*}
Let $u=\min\{1,q\}$. Then
	\begin{equation}\label{sigma1}
\sigma_n(S_{p,\theta}^rB,\mathcal{B}_{2m}^d)_q^u \leq \sup \limits_{(f\phi) \in S_{p,\theta}^rB(\R)} \, \inf_{\substack{\Lambda\subset \nabla \\ \#\Lambda = n}} \, \inf_{\substack{\mu_{\jb,\kb} \\ \mu_{\jb,\kb}=0 \, \text{if} \, (\jb,\kb)\not\in \Lambda}} \sum \limits_{\jb \in \N_{-1}^d} \Big\| \sum\limits_{\kb \in \Z} (\lambda_{2m;\jb,\kb}(f\phi) -\mu_{\jb,\kb}) \boldsymbol{s}_{2m;\jb,\kb}\Big\|^u_{q}.
\end{equation}
Recall that the basis function $\boldsymbol{s}_{2m;\jb,\kb}$ is defined as
$$
\boldsymbol{s}_{2m;\jb,\kb}(\boldsymbol{x})= \sum \limits_{\nb \in \Z} a_{\boldsymbol{n}}^{(m,e(\jb))} \boldsymbol{v}_{2m;\boldsymbol{j},\boldsymbol{k}+\nb}(\boldsymbol{x}),
$$
where the function $\boldsymbol{v}_{2m;\jb,\lb}$ is supported on the cube $K_{\jb,\lb}^{(m)}$. Then
$$
\boldsymbol{v}_{2m;\jb,\lb}(\x)=\boldsymbol{v}_{2m;\jb,\lb}(\x)\cdot\chi_{K_{\jb,\lb}^{(m)}}(\x) = \boldsymbol{v}_{2m;\jb,\lb}(\x) \cdot\sum\limits_{\tb\in \iota_m^{e(\jb)} } \chi_{I_{\jb,\lb+\tb}}(\x),
$$
where $\iota_m^{e(\jb)}:=\{\tb \in \zz_0^d: \,  t_i=0,...,2m-2 \text{ if } i \in e(\jb) \text{ and }  t_i=0,...,2m-1 \text{ if } i \not\in e(\jb) \}$ and
$I_{\jb,\lb+\tb}=\prod \limits_{i=1}^d I_{j_i,l_i+t_i}$, where
$$
I_{j,l+t}(x)=\begin{cases}
[2^{-j}(l+t),2^{-j}(l+t+1)], & j \in \N_0,\\
[l+t-m,l+t-m+1], & j=-1.
\end{cases}
$$
 Therefore, we have
\begin{align*}
\sum\limits_{\kb \in \Z} (\lambda_{2m;\jb,\kb}(f\phi) -\mu_{\jb,\kb}) &\boldsymbol{s}_{2m;\jb,\kb}(\x)= \sum\limits_{\kb \in \Z} (\lambda_{2m;\jb,\kb}(f\phi) -\mu_{\jb,\kb}) \sum \limits_{\nb \in \Z} a_{\boldsymbol{n}}^{(m,e(\jb))} \boldsymbol{v}_{2m;\boldsymbol{j},\boldsymbol{k}+\nb}(\boldsymbol{x})\\
&= \sum\limits_{\lb \in \Z} \bigg(\sum \limits_{\nb\in \Z} a_{\boldsymbol{n}}^{(m,e(\jb))} (\lambda_{2m;\jb,\lb-\nb}(f\phi) -\mu_{\jb,\lb-\nb}) \bigg) \boldsymbol{v}_{2m;\jb,\lb}(\boldsymbol{x})\\
&= \sum\limits_{\lb \in \Z} \bigg(\sum \limits_{\nb\in \Z} a_{\boldsymbol{n}}^{(m,e(\jb))} (\lambda_{2m;\jb,\lb-\nb}(f\phi) -\mu_{\jb,\lb-\nb}) \bigg) \boldsymbol{v}_{2m;\jb,\lb}(\boldsymbol{x}) \sum\limits_{\tb\in \iota_m} \chi_{I_{\jb,\lb+\tb}}(\x)\\
& \lesssim \sum\limits_{\lb \in \Z} \bigg(\sum \limits_{\nb\in \Z} a_{\boldsymbol{n}}^{(m,e(\jb))} (\lambda_{2m;\jb,\lb-\nb}(f\phi) -\mu_{\jb,\lb-\nb}) \bigg)  \sum\limits_{\tb\in \iota_m} \chi_{I_{\jb,\lb+\tb}}(\x)\\
&=\sum\limits_{\tb\in \iota_m} \sum\limits_{\lb \in \Z} \bigg(\sum \limits_{\nb\in \Z} a_{\boldsymbol{n}}^{(m,e(\jb))} (\lambda_{2m;\jb,\lb-\nb}(f\phi) -\mu_{\jb,\lb-\nb}) \bigg)   \chi_{I_{\jb,\lb+\tb}}(\x).
\end{align*}
From this we obtain for $q<\infty$. For $q=\infty$ modification is trivial
\begin{align*}
\bigg|\sum\limits_{\kb \in \Z} (\lambda_{2m;\jb,\kb}(f\phi) &-\mu_{\jb,\kb}) \boldsymbol{s}_{2m;\jb,\kb}(\x) \bigg|^q \\ &\lesssim \sum\limits_{\tb\in \iota_m} \bigg|\sum\limits_{\lb \in \Z} \bigg(\sum \limits_{\nb\in \Z} a_{\boldsymbol{n}}^{(m,e(\jb))} (\lambda_{2m;\jb,\lb-\nb}(f\phi) -\mu_{\jb,\lb-\nb}) \bigg)   \chi_{I_{\jb,\lb+\tb}}(\x) \bigg|^q\\
& \leq \sum\limits_{\tb\in \iota_m}\sum\limits_{\lb \in \Z}  \bigg| \sum \limits_{\nb\in \Z} a_{\boldsymbol{n}}^{(m,e(\jb))} (\lambda_{2m;\jb,\lb-\nb}(f\phi) -\mu_{\jb,\lb-\nb}) \bigg|^q   \chi_{I_{\jb,\lb+\tb}}(\x),
\end{align*}
which yields
 \begin{align*}
 \int\limits_{K} \bigg|\sum\limits_{\kb \in \Z} (\lambda_{2m;\jb,\kb}(f\phi) &-\mu_{\jb,\kb}) \boldsymbol{s}_{2m;\jb,\kb}(\x) \bigg|^q \, d \x \\ & \lesssim \sum\limits_{\tb\in \iota_m}\sum\limits_{\lb \in \Z}  \bigg| \sum \limits_{\nb\in \Z} a_{\boldsymbol{n}}^{(m,e(\jb))} (\lambda_{2m;\jb,\lb-\nb}(f\phi) -\mu_{\jb,\lb-\nb}) \bigg|^q   \int\limits_{K} \chi_{I_{\jb,\lb+\tb}}(\x) \,  d \x\\
 & \lesssim 2^{-|\jb|_1} \sum\limits_{\tb\in \iota_m}\sum\limits_{\lb \in \Z}  \bigg| \sum \limits_{\nb\in \Z} a_{\boldsymbol{n}}^{(m,e(\jb))} (\lambda_{2m;\jb,\lb-\nb}(f\phi) -\mu_{\jb,\lb-\nb}) \bigg|^q \\
& \lesssim 2^{-|\jb|_1}  \sum\limits_{\kb \in \Z} \bigg| \sum \limits_{\nb\in \Z} a_{\boldsymbol{n}}^{(m,e(\jb))} (\lambda_{2m;\jb,\kb}(f\phi) -\mu_{\jb,\kb}) \bigg|^q\\
& \lesssim 2^{-|\jb|_1}   \sum\limits_{\kb \in \Z} |\lambda_{2m;\jb,\kb}(f\phi) -\mu_{\jb,\kb}|^q.
 \end{align*}
Now we can proceed estimation of (\ref{sigma1})
$$
\sigma_n(S_{p,\theta}^rB,\mathcal{B}_{2m}^d)_q^u \leq \sup \limits_{(f\phi) \in S_{p,\theta}^rB(\R)} \, \inf_{\substack{\Lambda\subset \nabla \\ \#\Lambda = n}} \, \inf_{\substack{\mu_{\jb,\kb} \\ \mu_{\jb,\kb}=0 \, \text{if} \, (\jb,\kb)\not\in \Lambda}} \sum \limits_{\jb \in \N_{-1}^d}  2^{-|\jb|_1u/q}   \bigg(\sum\limits_{\kb \in \Z} |\lambda_{2m;\jb,\kb}(f\phi) -\mu_{\jb,\kb}|^q\bigg)^{u/q}.
$$
Since now $\supp \, (f\phi) =\Omega$ then by $\nabla_{\jb}(\Omega)$ we denote $\nabla_{\jb}(\Omega):=\{\kb \in\Z: \lambda_{2m;\jb,\kb}(f\phi)\neq 0\}$. The set $\nabla_{\jb}(\Omega)$ is finite for each $\jb$ since $\lambda_{2m;\jb,\kb}(f\phi)$ are defined via function values at dyadic points (see (\ref{exp-mul})).
By $\mathcal{D}$ we denote the canonical basis of unit vectors $\{e_{\jb,\kb}\}$, where $\jb \in \N_{-1}^d$ and $\kb \in \nabla_{\jb}(\Omega)$.

Proposition \ref{charact2} for $1/p<r<2m$ implies
\begin{align*}
\sigma_n(S_{p,\theta}^rB,\mathcal{B}_{2m}^d)_q^u & \leq \sup \limits_{a \in s_{p,\theta}^rb(\Omega)} \, \inf_{\substack{\Lambda\subset \nabla \\ \#\Lambda = n}} \, \inf_{\substack{\mu_{\jb,\kb} \\ \mu_{\jb,\kb}=0 \, \text{if} \, (\jb,\kb)\not\in \Lambda}} \sum \limits_{\jb \in \N_{-1}^d}  2^{-|\jb|_1u/q}   \bigg(\sum\limits_{\kb \in \Z} |a_{\jb,\kb} -\mu_{\jb,\kb}|^q\bigg)^{u/q}\\
&= \sigma_n (s_{p,\theta}^rb(\Omega),\mathcal{D})^u_{s_{q,u}^0 b}.
\end{align*}
From Corollary 5.11 \cite{HS2012} we obtain for $1/p-1/q<r\leq 1/\theta-1/u$
$$
\sigma_n (s_{p,\theta}^rb(\Omega),\mathcal{D})_{s_{q,u}^0 b} \lesssim n^{-r}.
$$
Therefore, for $1/p<r<\min\{2m,1/\theta-1/\min\{q,1\}\}$ or $1/p<r=1/\theta-1/\min\{q,1\}<2m$
$$
\sigma_n(S_{p,\theta}^rB(\R),\mathcal{B}_{2m}^d)_{q} \lesssim n^{-r}.
$$
\end{proof}

\begin{rem}\label{rem-best-1}
	The corresponding result for best $n$-term approximation with respect to the Faber-Schauder basis was obtained for the following range of smoothness parameter $1/p<r<\min\{1/\theta-1/\min\{q,1\},2\}$ or $1/p<r=1/\theta-1/\min\{q,1\}<2$ (see \cite[Theorem 6.21]{Glen2018} for details).
\end{rem}

\begin{satz} \label{best-m-term-2} Let $1/2m<p,q\leq \infty$, $0<\theta\leq \infty$ .
	\begin{itemize}
		\item [(i)]  For $\max\{1/p,1/\theta-1/\max\{q,1\}\}<r<2m$
		$$
	\sigma_n(S_{p,\theta}^rB,\mathcal{B}_{2m}^d)_q\lesssim n^{-r} (\log n)^{(d-1)(r-1/\theta+1)}
		$$
		holds.
		\item [(ii)]  For $\max\{1/p,1/\theta\}<r<2m$	
         $$
\sigma_n(S_{p,\theta}^rF,\mathcal{B}_{2m}^d)_q\lesssim n^{-r} (\log n)^{(d-1)(r-1/\theta+1)}
		$$
		holds.
	\end{itemize}
\end{satz}
\begin{proof}
We prove part (i) first. We start as in the proof of Theorem \ref{best-m-term-1}
\begin{align*}
\sigma_n(S_{p,\theta}^rB,\mathcal{B}_{2m}^d)_q&=\sup \limits_{f \in S_{p,\theta}^rB(\R)} \, \inf_{\substack{\Lambda\subset \nabla \\ \#\Lambda = n}} \, \inf \limits_{\mu_{\jb,\kb}} \Big\| f - \sum\limits_{(\jb,\kb)\in \Lambda} \mu_{\jb,\kb} \boldsymbol{s}_{2m;\jb,\kb} \Big\|_{q}\\
&\lesssim \sup \limits_{(f\phi) \in S_{p,\theta}^rB(\R)} \, \inf_{\substack{\Lambda\subset \nabla \\ \#\Lambda = n}} \, \inf_{\substack{\mu_{\jb,\kb} \\ \mu_{\jb,\kb}=0 \, \text{if} \, (\jb,\kb)\not\in \Lambda}} \Big\| \sum \limits_{\jb \in \N_{-1}^d}\sum\limits_{\kb \in \Z} (\lambda_{2m;\jb,\kb}(f\phi) -\mu_{\jb,\kb}) \boldsymbol{s}_{2m;\jb,\kb}\Big\|_{q}.
\end{align*}
Let us estimate the following norm
$$
\Big\| \sum \limits_{\jb \in \N_{-1}^d}\sum\limits_{\kb \in \Z} c_{\jb,\kb} \boldsymbol{s}_{2m;\jb,\kb}\Big\|_{q}.
$$

Since basis functions $\boldsymbol{s}_{2m;\jb,\kb}$ have smoothness $2m$, by using Lemma 4.3 from \cite{TV2019} we get for fixed $\jb$, $\x \in \R$
and $1/(2m)<\tau\leq 1$
$$
\sum \limits_{\kb \in \Z} |c_{\jb,\kb}|\, |\boldsymbol{s}_{2m;\jb,\kb}(\x)| \lesssim \Big(M\big(\sum\limits_{\kb \in \Z} c_{\jb,\kb} \chi_{I_{\jb,\kb}} \big)^{\tau}\Big)^{1/\tau}(\x),
$$
where $M$ is the Hardy-Littlewood maximal operator.
Using (\ref{m-ineq}) with $\theta=1$ and Lemma \ref{max-ineq} we obtain for $1/(2m)<\tau< \min\{q,1\}$
	\begin{equation}\label{add-est}
\Big\| \sum \limits_{\jb \in \N_{-1}^d}\sum\limits_{\kb \in \Z} c_{\jb,\kb} \boldsymbol{s}_{2m;\jb,\kb}\Big\|_{q} \lesssim \Big\| \sum\limits_{\jb \in \N_{-1}^d}\Big|\sum\limits_{\kb \in \Z} c_{\jb,\kb} \chi_{I_{\jb,\kb}}(\x) \Big| \, \Big\|_{q}.
	\end{equation}
In the last inequality $0<q<\infty$.

 Using (\ref{add-est}) and Proposition \ref{charact2} (i) for $1/p<r<2m$ we can write
 \begin{align*}
\sigma_n(S_{p,\theta}^rB,\mathcal{B}_{2m}^d)_q
 & \lesssim \sup \limits_{(f\phi) \in S_{p,\theta}^rB(\R)} \, \inf_{\substack{\Lambda\subset \nabla \\ \#\Lambda = n}} \, \inf_{\substack{\mu_{\jb,\kb} \\ \mu_{\jb,\kb}=0 \, \text{if} \, (\jb,\kb)\not\in \Lambda}} \Big\| \sum\limits_{\jb \in \N_{-1}^d}\Big|\sum\limits_{\kb \in \Z} (\lambda_{2m;\jb,\kb}(f\phi) -\mu_{\jb,\kb}) \chi_{I_{\jb,\kb}}(\x) \Big| \, \Big\|_{q}\\
 &\lesssim \sup \limits_{a \in s_{p,\theta}^rb(\Omega)} \, \inf_{\substack{\Lambda\subset \nabla \\ \#\Lambda = n}} \, \inf_{\substack{\mu_{\jb,\kb} \\ \mu_{\jb,\kb}=0 \, \text{if} \, (\jb,\kb)\not\in \Lambda}}\Big\| \sum\limits_{\jb \in \N_{-1}^d}\Big|\sum\limits_{\kb \in \Z} (a_{\jb,\kb} -\mu_{\jb,\kb}) \chi_{I_{\jb,\kb}}(\x) \Big| \, \Big\|_{q}\\
 &=\sigma_n(s_{p,\theta}^r b(\Omega),\mathcal{D})_{s_{q,1}^0 f}.
 \end{align*}

 In case when $q=\infty$ we consider similar technique as in  Theorem \ref{best-m-term-1}. Taking into account that $u=\min\{q,1\}$ we will get in the end the quantity $\sigma_n(s_{p,\theta}^r b(\Omega),\mathcal{D})_{s_{\infty,1}^0 b}$.

 Finally by using Corollary 5.8 from \cite{HS2012} we obtain that for $r>\max\left\{0, 1/\min\{p,\theta\} -1/\max\{1,q\} \right\}$
	\begin{equation}\label{sic-lem}
\sigma_n(s_{p,\theta}^r b(\Omega),\mathcal{D})_{s_{q,1}^0 y}\asymp n^{-r} (\log n )^{(d-1)(r-1/\theta+1)},
	\end{equation}
	where $y=\{b,f\}$.
That implies for $\max\{1/p,1/\theta-1/\max\{q,1\}\}<r<2m$
$$
\sigma_n(S_{p,\theta}^rB,\mathcal{B}_{2m}^d)_q \lesssim n^{-r} (\log n )^{(d-1)(r-1/\theta+1)}.
$$

To prove part (ii) we use the same technique with Proposition \ref{charact2} (ii) instead of Proposition \ref{charact2} (i).
\end{proof}

\begin{rem}\label{rem-best-2}
	We refer again to \cite{Glen2018} (see Theorem 6.23), where the corresponding estimates with respect to Faber-Schauder basis can be found for $\max\{1/p,1/\theta-1/\max\{q,1\}\}<r<2$ and $\max\{1/p,1/\theta\}<r<2$ for Besov and Triebel-Lizorkin spaces respectively.
\end{rem}

\begin{rem}\label{rem-best-4.5}
We compare results for Besov spaces from Theorems \ref{best-m-term-1} and \ref{best-m-term-2} with results that where obtained for approximation by linear sampling methods in the paper \cite{DD2011}. For the case $p<q$ the rate of decay of the ``main term'' for this quantity $r_n$ is $n^{-r+1/p-1/q}$, what is a worse error decay than for the nonlinear method we consider here, where it is $n^{-r}$. See also Paragraph 4.3.
\end{rem}

\subsection{Discussion and special cases.}
In this subsection we consider the Sobolev spaces of mixed smoothness $S_{p}^{r}W$. Particular interest has the next theorem, where  we obtain the upper estimate for best $n$-term approximation of spaces $S_{p}^{r}W$ for the limiting smoothness $r=2m$. Note that in the proof we use ideas that were offered in \cite{Glen2018} for the similar problem with respect to Faber-Schauder basis for function spaces $S_{p}^{2}W$.
\begin{satz} \label{best-m-term-3} Let $1<p<\infty$ and $0<q\leq \infty$. Then 	
	\begin{equation}\label{m-term-3}
	\sigma_n(S_{p}^{2m}W,\mathcal{B}_{2m}^d)_q\lesssim n^{-2m} (\log^{d-1}n)^{2m+1}.
	\end{equation}
\end{satz}
\begin{proof}
	Since the compactly supported $f\phi \in S_{p}^{2m}W(\Omega)$ can be represented by the series (\ref{exp-mul}), we have
	\begin{align*}
	\sigma_n(S_{p}^{2m}W,\mathcal{B}_{2m}^d)_q&=\sup \limits_{f \in S_{p}^rW(\R)} \, \inf_{\substack{\Lambda\subset \nabla \\ \#\Lambda = n}} \, \inf \limits_{\mu_{\jb,\kb}} \Big\| f - \sum\limits_{(\jb,\kb)\in \Lambda} \mu_{\jb,\kb} \boldsymbol{s}_{2m;\jb,\kb} \Big\|_{q}\\
	&= \sup \limits_{ f\phi \in S_{p}^{2m}W(\R)} \, \inf_{\substack{\Lambda\subset \nabla \\ \#\Lambda = n}} \, \inf_{\substack{\mu_{\jb,\kb} \\ \mu_{\jb,\kb}=0 \, \text{if} \, (\jb,\kb)\not\in \Lambda}} \Big\| \sum \limits_{\jb \in \N_{-1}^d}\sum\limits_{\kb \in \Z} (\lambda_{2m;\jb,\kb}(f\phi) -\mu_{\jb,\kb}) \boldsymbol{s}_{2m;\jb,\kb}\Big\|_{q}.
	\end{align*}	
	By using again (\ref{add-est}) we can write
	$$
	\sigma_n(S_{p}^{2m}W,\mathcal{B}_{2m}^d)_q \lesssim \sup \limits_{f\phi \in S_{p}^rW(\R)} \, \inf_{\substack{\Lambda\subset \nabla \\ \#\Lambda = n}} \, \inf_{\substack{\mu_{\jb,\kb} \\ \mu_{\jb,\kb}=0 \, \text{if} \, (\jb,\kb)\not\in \Lambda}}   \Big\| \sum\limits_{\jb \in \N_{-1}^d}\Big|\sum\limits_{\kb \in \Z} (a_{\jb,\kb} -\mu_{\jb,\kb}) \chi_{I_{\jb,\kb}}(\x) \Big| \, \Big\|_{q}.
	$$
	
	Further we prove the following inequality for $f \in S_p^{2m} W(\R)$ and $1<p<\infty$
	\begin{equation} \label{w-b}
	\sup\limits_{\jb \in \N_{-1}^d} 2^{2m|\jb|_1} \Big\|\sum\limits_{\kb \in \Z} \lambda_{2m;\jb,\kb}(f) \chi_{I_{\jb,\kb}}  \Big \|_p \lesssim \|f\|_{S_{p}^{2m}W}.
	\end{equation}
	In \cite[P. 56]{Glen2018} it is shown that it is enough to prove this inequality for functions $\varphi \in D(\R)$ since this space is dense in $S_p^{2m}W$.
	For the simplicity we consider the case $d=2$. The supremum $\sup\limits_{\jb \in \N_{-1}^2}$ can be decomposed
	$$
	\sup\limits_{\jb \in \N_{-1}^2}\leq \sup_{\substack{j_i \in \N_0 \\ i=1,2}} +\sup_{\substack{j_1=-1 \\ j_2 \in \N_0}}+\sup_{\substack{j_1 \in \N_0 \\ j_2=-1}}+\sup_{\substack{j_1=-1 \\ j_2=-1}}.
	$$
	
	We will further use the following lemma.
	\begin{lem}\label{sic-ineq} \cite[Proposition 2]{SC2000}
		Let $0<p<\infty$ and $\widetilde{p}=\min\{1,p\}$. Then there exists a constant $C$ such that
		$$
		\Big(\sum_{k \in \zz} |f(k)|^p \Big)^{1/p}\leq C \|f\|_{B_{p,\widetilde{p}}^{1/p}},
		$$
		holds for all $f \in B_{p,\widetilde{p}}^{1/p}$.
	\end{lem}
	
	We have
	\begin{align*}
	\sup_{\substack{j_1=-1\\ j_2=-1}} 2^{2m|\jb|_1} \Big\|\sum\limits_{\kb \in \zz^2} \lambda_{2m;\jb,\kb}(\varphi) \chi_{I_{\jb,\kb}}  \Big \|_p & = C(m) \Big\|\sum\limits_{k_1 \in \zz} \sum\limits_{k_2 \in \zz} \varphi(k_1,k_2) \chi_{-1,k_1} \chi_{-1,k_2}  \Big\|_p\\
	& \lesssim \Big( \sum\limits_{k_1 \in \zz} \sum\limits_{k_2 \in \zz} |\varphi(k_1,k_2)|^p \Big)^{1/p}.
	\end{align*}
	According to Lemma \ref{sic-ineq} we get
	$$
	\sup_{\substack{j_1=-1\\ j_2=-1}} 2^{2m|\jb|_1} \Big\|\sum\limits_{\kb \in \zz^2} \lambda_{2m;\jb,\kb}(\varphi) \chi_{I_{\jb,\kb}}  \Big \|_p \lesssim  \Big( \sum\limits_{k_1 \in \zz} \|\varphi(k_1,\cdot)\|^p_{B_{p,1}^{1/p}} \Big)^{1/p}.
	$$
	Using the embedding for $p>1$
	$$
	W_p^{2m}\subset B^{2m}_{p,\max\{p,2\}} \subset B_{p,1}^{1/p},
	$$
	we proceed
	$$
	\sup_{\substack{j_1=-1\\ j_2=-1}} 2^{2m|\jb|_1} \Big\|\sum\limits_{\kb \in \zz^2} \lambda_{2m;\jb,\kb}(\varphi) \chi_{I_{\jb,\kb}}  \Big \|_p \lesssim  \Big( \sum\limits_{k_1 \in \zz} \|\varphi(k_1,\cdot)\|^p_{W_p^{2m}} \Big)^{1/p}.
	$$
	
	Repeating the same steps with respect to the first variable we obtain
	\begin{align*}
	\sup_{\substack{j_1=-1\\ j_2=-1}} 2^{2m|\jb|_1} \Big\|\sum\limits_{\kb \in \zz^2} \lambda_{2m;\jb,\kb}(\varphi) \chi_{I_{\jb,\kb}}  \Big \|_p \lesssim & \Big\| \|\varphi\|_{W_p^{2m}} \Big\|_{B_{p,1}^{1/p}} \\
	& \leq  \Big\| \|\varphi\|_{W_p^{2m}} \Big\|_{W_p^{2m}}= \|\varphi\|_{S_p^{2m}W}.
	\end{align*}
	
	Let now $j_1=-1$ and $j_2 \in \N_0$. In this case
	$$
	\lambda_{2m;\jb,\kb}(\varphi) = \sum\limits_{s=0}^{2m-2} (-1)^s N_{2m}(s+1) \Delta^{2m}_{2^{-j_2-1},2} \varphi\Big(k_1, \frac{2k_2+s}{2^{j_2+1}}\Big).
	$$
	We have
	\begin{align}
	\sup_{\substack{j_1=-1\\ j_2 \in \N_0}} 2^{2m|\jb|_1} \Big\|\sum\limits_{\kb \in \zz^2} \lambda_{2m;\jb,\kb}(\varphi) \chi_{I_{\jb,\kb}}  \Big \|^p_p=&\sup\limits_{j_2 \in N_0} C(m) 2^{2mp j_2} \notag \\
	&\times  \int\limits_{\re} \int\limits_{\re} \Big| \sum\limits_{k_1 \in \zz} \sum\limits_{k_2 \in \zz} \lambda_{2m;\jb,\kb}(\varphi)  \chi_{-1,k_1}(x_1) \chi_{j_2,k_2}(x_2)\Big|^p dx_1 dx_2 \notag \\
	 \lesssim &\sup\limits_{j_2 \in N_0}  2^{2mp j_2} \sum\limits_{k_1 \in \zz} \int\limits_{\re} \Big| \sum\limits_{k_2 \in \zz} \lambda_{2m;\jb,\kb}(\varphi)   \chi_{j_2,k_2}(x_2)\Big|^p dx_2 \notag\\
	 \leq &\sup\limits_{j_2 \in N_0}  2^{2mp j_2} 2^{-j_2} \sum\limits_{k_1 \in \zz} \sum\limits_{k_2 \in \zz} |\lambda_{2m;\jb,\kb}(\varphi)|^p. \label{est-j}
	\end{align}
	According to inequality (7.10) \cite{DL93}
	$$
	\Delta_h^r\varphi(x)=h^{r-1} \int\limits_{\re} \varphi^{(r)} N_r(h^{-1}(t-x)) dt,
	$$
	where $N_r$ is $r$th B-spline supported on $[0,r]$.
	Therefore,
	$$
	\Delta^{2m}_{2^{-j_2-1}} g\Big(\frac{2k_2+s}{2^{j_2+1}}\Big)=2^{-2mj_2}2^{-2m}2^{j_2+1} \int\limits_{2^{-j_2-1}(2k_2+s)}^{2^{-j_2-1}(2m+2k_2+s)} g^{(2m)}(t) N_{2m}(2^{j_2+1}t-2k_2-s) dt.
	$$
	By using H\"{o}lder inequalities we get
	\begin{align}
	\Big| \Delta^{2m}_{2^{-j_2-1}} g\Big(\frac{2k_2+s}{2^{j_2+1}}\Big)\Big|^p &\lesssim 2^{-2mj_2 p} 2^{j_2} \int\limits_{2^{-j_2-1}(2k_2+s)}^{2^{-j_2-1}(2m+2k_2+s)} |g^{(2m)}(t)|^p dt\notag\\
	&=2^{-2mj_2 p} 2^{j_2} \sum \limits_{i=s}^{s+2m-1} \int\limits_{2^{-j_2-1}(2k_2+i)}^{2^{-j_2-1}(2k_2+1+i)}|g^{(2m)}(t)|^p dt. \label{delta-in}
	\end{align}
	Now we can continue estimation of (\ref{est-j})
	\begin{align*}
	\sup_{\substack{j_1=-1\\ j_2 \in \N_0}} 2^{2m|\jb|_1}& \Big\|\sum\limits_{\kb \in \zz^2} \lambda_{2m;\jb,\kb}(\varphi) \chi_{I_{\jb,\kb}}  \Big \|^p_p\\& \lesssim  \sup\limits_{j_2 \in N_0}  2^{2mp j_2} 2^{-j_2} \sum\limits_{k_1 \in \zz} \sum\limits_{k_2 \in \zz} \Big|\sum\limits_{s=0}^{2m-2} (-1)^s N_{2m}(s+1) \Delta^{2m}_{2^{-j_2-1},2} \varphi\Big(k_1, \frac{2k_2+s}{2^{j_2+1}}\Big)\Big|^p\\
	& \leq \sup\limits_{j_2 \in N_0}  2^{2mp j_2} 2^{-j_2} \sum\limits_{s=0}^{2m-2} | N_{2m}(s+1)| \sum\limits_{k_1 \in \zz} \sum\limits_{k_2 \in \zz} \Big|\Delta^{2m}_{2^{-j_2-1},2} \varphi\Big(k_1, \frac{2k_2+s}{2^{j_2+1}}\Big)\Big|^p\\
	& \lesssim \sup\limits_{j_2 \in N_0} \sum\limits_{s=0}^{2m-2} | N_{2m}(s+1)| \sum \limits_{i=s}^{s+2m-1} \sum\limits_{k_1 \in \zz} \sum\limits_{k_2 \in \zz}\int\limits_{2^{-j_2-1}(2k_2+i)}^{2^{-j_2-1}(2k_2+1+i)}|\varphi^{(0,2m)}(k_1,t)|^p dt \\
	& \leq \sup\limits_{j_2 \in N_0} \sum\limits_{s=0}^{2m-2} | N_{2m}(s+1)| \sum \limits_{i=s}^{s+2m-1} \sum\limits_{k_1 \in \zz} \int\limits_{\re} |\varphi^{(0,2m)}(k_1,t)|^p dt \\
	& \lesssim \sum\limits_{k_1 \in \zz} \int\limits_{\re} |\varphi^{(0,2m)}(k_1,t)|^p dt.
	\end{align*}
	
	Using Lemma \ref{sic-ineq} we proceed
	$$
	\Big\|\sum\limits_{k_1\in\zz}|\varphi^{(0,2m)}(k_1,t)|^p \Big\|_p^p \lesssim \Big\| \|\varphi^{(0,2m)}(\cdot,t)\|^p_{B_{p,1}^{1/p}} \Big\|_p^p\leq \Big\| \|\varphi^{(0,2m)}(\cdot,t)\|^p_{W_p^{2m}} \Big\|_p^p=\|\varphi\|_{S_p^{2m}W}.
	$$
	
	The case $j_1\in N_0$, $j_2=-1$ works analogously and for $\jb \in \N_0^2$ we use the inequality (\ref{delta-in}) in both directions.
	
	Note that the left hand side of (\ref{w-b}) is the norm of the sequences space $s_{p,\infty}^{2m}b$. Therefore,
	\begin{align*}
\sigma_n(S_{p}^{2m}W,\mathcal{B}_{2m}^d)_q
	&\lesssim \sup \limits_{a \in s_{p,\infty}^{2m}b(\Omega)} \, \inf_{\substack{\Lambda\subset \nabla \\ \#\Lambda = n}} \, \inf_{\substack{\mu_{\jb,\kb} \\ \mu_{\jb,\kb}=0 \, \text{if} \, (\jb,\kb)\not\in \Lambda}} \Big\| \sum\limits_{\jb \in \N_{-1}^d}\Big|\sum\limits_{\kb \in \Z} (a_{\jb,\kb} -\mu_{\jb,\kb}) \chi_{I_{\jb,\kb}}(\x) \Big| \, \Big\|_{q}\\
	&=\sigma_n(s_{p,\infty}^{2m} b(\Omega),\mathcal{D})_{s_{q,1}^0 f}.
	\end{align*}
	
	In the case when $q=\infty$ we consider similar technique as in  Theorem \ref{best-m-term-1}. Taking into account that $u=\min\{q,1\}$ we will get in the end the quantity $\sigma_n(s_{p,\infty}^{2m} b(\Omega),\mathcal{D})_{s_{\infty,1}^0 b}$.
		Using (\ref{sic-lem}) we get for $p>1$ the inequality (\ref{m-term-3}).
\end{proof}

Now we consider $L_\infty$ metric. As a corollary from Theorem \ref{best-m-term-2} we get
\begin{cor}\label{cor-1}
	For $1<p<\infty$ and $\max\{1/p,1/2\}<r<2m$ the following inequality holds
$$
\sigma_n(S_p^rW,\mathcal{B}_{2m}^d)_\infty \lesssim n^{-r} (\log^{d-1}n)^{r+\frac{1}{2}}.
$$
\end{cor}

As a particular case of Theorem \ref{best-m-term-3} we get the following results for Sobolev spaces of limiting smoothness $r=2m$. Note, that there is jump in the exponent of the logarithm. We do not know whether this is sharp.
\begin{cor}\label{cor-2}
Let $1<p<\infty$. Then 	
$$
\sigma_n(S_{p}^{2m}W,\mathcal{B}_{2m}^d)_\infty\lesssim n^{-2m} (\log^{d-1}n)^{2m+1}.
$$
\end{cor}

\begin{rem}\label{low-bound}
Note that in \cite{Glen2018} the author also obtains lower estimates for best $n$-term approximation with respect to Faber-Schauder basis in some specific cases (see Theorem 6.20 in \cite{Glen2018}). We do not obtain the analog of this results with respect to the Faber spline basis. The main difficulties in the proof come from the fact that our basis functions are not compactly supported, although well localized.
\end{rem}

\subsection{Level-wise greedy algorithms} To obtain the upper estimates in Theorems \ref{best-m-term-1}, \ref{best-m-term-2} and \ref{best-m-term-3} we use estimates for the best $n$-term approximation of sequence spaces from Hansen and Sickel \cite{HS2012}. Together with the question of obtaining estimates for the quantity (\ref{best-n}), the question whether there is a constructive method is interesting. Is there a method that defines $\Lambda \subset \Omega$, $\# \Lambda=n$ and coefficients $c_\alpha$, $\alpha \in \Lambda$, such that the aggregate  $\sum_{\alpha\in \Lambda} c_\alpha u_\alpha$ realizes this estimate? For the case of ``large smoothness'' the corresponding constructive method was known already from the paper \cite{HS2012}. It is a level-wise greedy algorithm, see also \cite[6.3, Algo.\ 1]{Glen2018}. An analogous  algorithm will also work for the higher order Faber spline basis.  The idea of this level-wise greedy method is rather simple. We always consider a linear hyperbolic cross/sparse grid projection in an initial step (see for instance \cite{DD2011}) and adaptively throw away all those coefficients whose contribution is to small. In each level the theory predicts how many coefficients we can keep in order not to destroy the rate of convergence. Of course, we keep the largest coefficients in each level. In the ``large smoothness setting'' this number only depends on the level (or hyperbolic layer). We refer to \cite{HS2012} and \cite{Glen2018} for further details. In case of ``small smoothness'' a constructive algorithm has been given recently in \cite[6.3, Algo.\ 2]{Glen2018}. It is a level-wise greedy algorithm but with a higher amount of ``nonlinearity''. There, the number of coefficients to keep in each layer does not only depend on the respective layer, it also depends on the contribution of the entire layer to the norm. So the whole budget $n$ of coefficients to keep is proportionally distributed among the different initial layers according to their contribution to the norm. If this contribution is large we keep many coefficients of this level, if the contribution is small the level will be neglected.

In any case, the whole sense of this greedy operation is to reduce the amount of data needed to represent the function. This can be seen as a compression step. When data has to be transmitted or stored this will reduce the amount of storage space significantly.

A further interesting question is whether the so called ``pure greedy algorithm'' (see, for example, \cite{Tem2011}) will also give the same order of decay. This question is not investigated yet even for the Faber-Schauder basis.

\appendix
\section{Definitions and auxiliary statements}

\subsection{Definition of Besov and Triebel-Lizorkin spaces with mixed smoothness}
Let $S(\R)$ be the Schwartz space of infinitely times differentiable fast decreasing functions. By $S'(\R)$ we denote the topological dual of $S(\R)$ that is the space of tempered distributions.

In the sequel we will define Besov-Triebel-Lizorkin spaces of mixed smoothness on $\R$ via local mean kernels. Classically, they are defined via a smooth dyadic decomposition on the Fourier image, which represents a special case as we will point out below. For comments about the history of this characterization, we refer to \cite[Rem\ 4.6]{UU2015}. Let $\Psi_0, \Psi_1 \in S(\R)$ such that for some $\varepsilon>0$: 1) $|\mathcal{F}\Psi_0(\xi)|>0$ for $|\xi|<\varepsilon$; 2) $|\mathcal{F}\Psi_1(\xi)|>0$ for $\varepsilon/2<|\xi|<2\varepsilon$; 3) $D^{\alpha} \mathcal{F}\Psi_1(0)=0$ for $0\leq \alpha <L$. Then for $j\geq 1$ put
$$
\Psi_j(x):=2^{j-1}\Psi_1(2^{j-1}x).
$$
In other words, $\Psi_j$, $j\geq 1$, satisfies the $L$-th order moment condition hold, i.e. for all $0\leq \alpha<L$
$$
\int\limits_{\re} x^\alpha \Psi_j(x)dx=0.
$$
For $\jb \in \N_0^d$ we define $\Psi_{\jb}(\x)=\prod\limits_{i=1}^d \Psi_{j_i}(x_i)$ for $\x \in \R$.

\begin{defi}\label{local-mean} \cite{Vy06,T08_2}
	Let $0< p,\theta \leq \infty$ ($p<\infty$ for $F$-case), $\{\Psi_{\jb}\}_{\jb \in \N_0^d}$ as defined above with $L>r$. Then
	$$
	\|f\|_{S_{p,\theta}^rB(\R)}\asymp
	\begin{cases}
	\Big(\sum\limits_{\jb \in \N_0^d} 2^{\theta r |\jb|_1} \|\Psi_{\jb}*f\|_p^\theta   \Big)^{1/\theta}, & 0< \theta <\infty, \\
	\sup \limits_{\jb \in \N_0^d}2^{r |\jb|_1} \|\Psi_{\jb}*f\|_p, & \theta=\infty,
	\end{cases}
	$$
	and
	$$
	\|f\|_{S_{p,\theta}^rF(\R)}\asymp
	\begin{cases}
	\Big\|\Big(\sum\limits_{\jb \in \N_0^d} 2^{\theta r |\jb|} |\Psi_{\jb}*f|^\theta   \Big)^{1/\theta}\Big\|_p, & 0< \theta <\infty, \\
	\Big\|\sup \limits_{\jb \in \N_0^d}2^{r |\jb|_1} |\Psi_{\jb}*f|\Big\|_p, & \theta=\infty,
	\end{cases}
	$$	
\end{defi}
Clearly, this definition is independent of the choice of the kernel. Different kernels produce different equivalent quasi-norms. Of particular interest is the classical decomposition of unity, constructed as follows. Let $\varphi \in S(\re)$ but compactly supported such that $\varphi(x) = 1$ if $|x|\leq 1$ and $\varphi(x) = 0$ if $|x|>2$. Then we put $\Psi_0 = \cf^{-1}\varphi$ and $\Psi_1:=2\Psi_0(2\cdot)-\Psi_0$. After tensorization we define
$$
	\delta_{\jb}[f]:=\Psi_{\jb}\ast f\,.
$$
The crucial feature of this definition is the fact that now every $f \in S'(\R)$ can be decomposed as
\begin{equation} \label{frep}
f=\sum \limits_{\jb \in \N_0^d} \delta_{\jb}[f]
\end{equation}
with convergence in $S'(\R)$.

Further we define discrete function spaces $s_{p,\theta}^rb$ and $s_{p,\theta}^rf$.
For $j \in \N_{-1}$ and $k \in \zz$ we define the intervals
$$
I_{j,k}:=
\begin{cases}
[2^{-j}k,2^{-j}(k+1)], & j \geq 0, \\
[k-1/2,k+1/2], & j=-1,
\end{cases}
$$
then for $\jb\in \N_{-1}^d$ and $\kb\in\Z$ we put $I_{\jb,\kb}=\prod\limits_{i=1}^d I_{j_i,k_i}$.
Then the corresponding characteristic functions
$$
\chi_{\jb,\kb}(\x):=
\begin{cases}
1, & \x \in I_{\jb,\kb},\\
0, & \text{otherwise}.
\end{cases}
$$

\begin{defi}\label{besov-seq}
	Let $r \in \re$ and $0< p,\theta \leq \infty$.	 By $s_{p,\theta}^rb$ and $s_{p,\theta}^rf$ ($p<\infty$ for $f$-case) we define the spaces of sequences of coefficients $\lambda=(\lambda_{\jb,\kb})_{\jb \in \N_{-1}^d, \kb \in \Z}$ with the finite norms
	$$
	\|\lambda\|_{s_{p,\theta}^rb}:=
	\begin{cases}
	\Big(\sum\limits_{\jb \in \N_{-1}^d} 2^{\theta r |\jb|_1} \Big\|\sum\limits_{\kb \in \Z} \lambda_{\jb,\kb} \chi_{\jb,\kb} \Big\|_p^\theta \Big)^{1/\theta},& 0< \theta <\infty, \\
	\sup \limits_{\jb \in \N_{-1}^d}2^{r |\jb|_1} \Big\|\sum\limits_{\kb \in \zz} \lambda_{\jb,\kb} \chi_{\jb,\kb} \Big\|_p, & \theta=\infty.
	\end{cases}
	$$
	and
	$$
	\|\lambda\|_{s_{p,\theta}^rf}:=
	\begin{cases}
	\Big\| \Big(\sum\limits_{\jb \in \N_{-1}^d} 2^{\theta r |\jb|_1} \Big|\sum\limits_{\kb \in \Z} \lambda_{\jb,\kb} \chi_{\jb,\kb} \Big|^\theta \Big)^{1/\theta}\Big\|_p,& 0< \theta <\infty, \\
	\Big\| \sup \limits_{\jb \in \N_{-1}^d}2^{r |\jb|_1} \Big| \sum\limits_{\kb \in \Z} \lambda_{\jb,\kb} \chi_{\jb,\kb} \Big| \Big\|_p, & \theta=\infty.
	\end{cases}
	$$
	respectively.
\end{defi}

\subsection{Maximal inequalities}
\begin{defi}\label{peetre}
	Let $\boldsymbol{b}>0$ and $a>0$. Then for $f \in L_1(\R)$ with $\mathcal{F}f$ compactly supported  we define the Peetre maximal operator
	$$
	P_{\boldsymbol{b},a}f(x):=\sup\limits_{\boldsymbol{y} \in \R}\dfrac{|f(\boldsymbol{y})|}{(1+b_1|x_1-y_1|)^{a}\cdot \ldots \cdot (1+b_d|x_d-y_d|)^{a}}.
	$$
\end{defi}

\begin{defi}\label{hardy}
	For a locally integrable function $f: \R\rightarrow \mathbb{C}$ we define the Hardy-Littlewood maximal operator defined by
	$$
	(Mf)(\x)=\sup\limits_{\x \in Q}\frac{1}{Q}\int\limits_{Q}|f(\boldsymbol{y})| d\boldsymbol{y}, \ \ \x \in \re,
	$$
	where the supremum is taken over all segments that contain $\x$.
\end{defi}

\begin{lem}\label{max-ineq}
	For $1<p<\infty$ and $1<q\leq \infty$ there exists a constant $c>0$ such that
	$$
	\Big\|\Big(\sum\limits_{\lb \in I}|Mf_{\lb}|^q \Big)^{1/q} \Big\|_p\leq c \Big\|\Big(\sum\limits_{\lb \in I}|f_{\lb}|^q \Big)^{1/q} \Big\|_p
	$$
	holds for all sequences $\{f_{\lb}\}_{\lb \in I}$ of locally Lebesgue integrable functions on $\R$.
\end{lem}

\begin{lem}\label{peetre-ineq} \cite[Lemma 3.3.1]{TDiff06}
	Let $a,b>0$, $m\in \N$, $h \in \re$ and $f\in L_1(\re)$ with $\supp \mathcal{F}f\subset [-b,b]$. Then there exists a constant $C>0$ independent of $f,b$ and $h$ such that
	$$
	|\Delta_h^m f(x)|\leq C \min\{1,|bh|^m\} \max\{1,|bh|^a\} P_{b,a}f(x)
	$$
	holds.
\end{lem}

\begin{lem}\label{peetre-ineq1}\cite[1.6.4]{SchTrieb87}
	Let $0< p \leq \infty$,  $\boldsymbol{b}>0$ and $a>1/p$. For a bandlimited function $f \in L_1(\R)$ with $\supp \mathcal{F}f\subset [-\boldsymbol{b},\boldsymbol{b}]$	the following inequality holds
	$$
	\|P_{\boldsymbol{b},a}f\|_p	\leq C \|f\|_p,
	$$	
	where $C$ is some constant independent on $f$ and $\boldsymbol{b}$.
\end{lem}

\begin{lem}\label{peetre-ineq2}\cite[1.6.4]{SchTrieb87}
	Let $0< p,\theta \leq \infty$, $p\neq \infty$, $\boldsymbol{b}^{\lb}>0$ for $\lb \in I$ and $a>\max\{1/p,1/\theta\}$. There exists a constant such that for all systems of functions $\{f_{\lb}\}_{\lb \in I}$ with $\supp \mathcal{F}f_{\lb}\subset [-\boldsymbol{b}^{\lb},\boldsymbol{b}^{\lb}]$ the following inequality holds
	$$
	\Big\|\Big( \sum\limits_{\lb \in I}|P_{\boldsymbol{b}^{\lb},a}f_{\lb}|^\theta \Big)^{1/\theta} \Big\|_p\leq C \Big\|\Big( \sum\limits_{\lb \in I}|f_{\lb}|^\theta \Big)^{1/\theta} \Big\|_p.
	$$	
\end{lem}

\subsection{Mixed differences}

For a function $f$ defined on $\R$ we define the first order difference with step $h$ in direction $i \in [d]$
$$
\Delta_h^{1,i}f(x):=f(x+h \boldsymbol{e}_i)-f(x),
$$
where $\boldsymbol{e}_i$ is unit basis vector with $i$-th coordinate equal 1. Then $m$-th order difference is defined iteratively by
$$
\Delta_h^{m,i}f(x):=\Delta_h^{1,i}\Delta_h^{m-1,i}f(x).
$$
For each subset $e\subset [d]$ and a step-vector $\boldsymbol{h} \in \R$ we defined $m$-th order mixed difference operator by
$$
\Delta_{\boldsymbol{h}}^{m,e}f(\x):=\left(\prod\limits_{i\in e} \Delta_{h_i}^{m,i} \right)f(\x).
$$

\section{Examples of the construction of the basis functions}
In this section we present examples of the construction of higher order Faber splines. Originally this idea goes to \cite{DU2019}. Here we present examples of the construction of basis functions $s_{2m;j,k}$ for cases $m=2$ and $m=3$.

\subsection{The case $m=2$} The approach in \cite{DU2019} leads to the following coefficients
$$
a_n^{(2)}=\begin{cases}
(-6-4\sqrt{3})(-2-\sqrt{3})^{n-1}+(6+7\sqrt{3}/2)(7+4\sqrt{3})^{n-1} & \text{if} \ \ n\leq 1, \\
(6-4\sqrt{3})(-2+\sqrt{3})^{n-1}+(-6+7\sqrt{3}/2)(7-4\sqrt{3})^{n-1} & \text{if} \ \ n>1,
\end{cases}
$$
or we can compute them also numerically
\begin{align*}
a_0^{(2)}&=4.33,  &a_{\pm 1}^{(2)}&=-0.866,  &a_{\pm 2}^{(2)}&=0.253,  &a_{\pm 3}^{(2)}&=-6.6 \cdot 10^{-2},  &a_{\pm 4}^{(2)}&=1.8\cdot 10^{-2},\\
a_{\pm 5}^{(2)}&=-4.7\cdot 10^{-3},  &a_{\pm 6}^{(2)}&=1.3\cdot 10^{-3},  &a_{\pm 7}^{(2)}&=-3.4\cdot 10^{-4},  &a_{\pm 8}^{(2)}&=9.2\cdot 10^{-5},  &a_{\pm 9}^{(2)}&=-2.4\cdot 10^{-5}.
\end{align*}
Then according to formulas (\ref{b-basis}) we get
$$
s_{4;j,k}(x)= \sum \limits_{n \in \zz} a_n^{(2)}  v_4(2^jx-k-n),
$$
where
$$
v_4(t)=\frac{1}{36}
\begin{cases}
t^3, & 0\leq t \leq 1/2, \\
1-6t+12t^2-7t^3, & 1/2<t \leq 1, \\
-22+63t-57t^2+16t^3, & 1<t\leq 3/2, \\
86-153t+87t^2-16t^3, & 3/2<t\leq 2, \\
-98+123t-51t^2+7t^3, & 2<t\leq 5/2, \\
27-27t+9t^2-t^3, & 5/2<t\leq 3,\\
0, & \text{otherwise}.
\end{cases}
$$

\subsection{The case $m=3$} In this case the polynomial (6.2) in \cite{DU2019} has the following representation
$$
1-518z-11072z^2+41734z^3+170110z^4+41734z^5-11072z^6-518z^7+z^8=0.
$$
There are the following algebraic roots
\begin{align*}
z_0&= 136+13\sqrt{105}-4\sqrt{2265+221\sqrt{105}},  & z_1=\frac{1}{2}\left(-13-\sqrt{105}+\sqrt{2(135+13\sqrt{105})}\right),\\
z_2&=\dfrac{1}{2}\left(-13+\sqrt{105}+\sqrt{2(135-13\sqrt{105})}\right), & z_3=136-13\sqrt{105}-4\sqrt{2265-221\sqrt{105}}, \\
z_4&= 136+13\sqrt{105}+4\sqrt{2265+221\sqrt{105}},  & z_5=\frac{1}{2}\left(-13-\sqrt{105}-\sqrt{2(135+13\sqrt{105})}\right),\\
z_6&=\dfrac{1}{2}\left(-13+\sqrt{105}-\sqrt{2(135-13\sqrt{105})}\right), &  z_7=136-13\sqrt{105}+4\sqrt{2265-221\sqrt{105}}.
\end{align*}
Plugging this into Formula (6.4) in \cite{DU2019} leads to the following numerical values for the coefficients.
\begin{align*}
a_0^{(3)}&=12.251,  &a_{\pm 1}^{(3)}&=-3.765,  &a_{\pm 2}^{(3)}&=1.921,  &a_{\pm 3}^{(3)}&=-0.772 ,  &a_{\pm 4}^{(3)}&=0.343,\\
a_{\pm 5}^{(3)}&=-0.145,  &a_{\pm 6}^{(3)}&=6.3\cdot 10^{-2},  &a_{\pm 7}^{(3)}&=-2.7\cdot 10^{-2},  &a_{\pm 8}^{(3)}&=1.1\cdot 10^{-2},  &a_{\pm 9}^{(3)}&=-5.02\cdot 10^{-3},\\
a_{\pm 10}^{(3)}&=2.1\cdot 10^{-3},  &a_{\pm 11}^{(3)}&=-9.3\cdot 10^{-4},  &a_{\pm 12}^{(3)}&=4.01\cdot 10^{-4},  &a_{\pm 13}^{(3)}&=-1.7\cdot 10^{-4},  &a_{\pm 14}^{(3)}&=7.04\cdot 10^{-5},
\end{align*}
and basis functions $s_{6;j,k}$ are defined as follows
$$
s_{6;j,k}(x)= \sum \limits_{n \in \zz} a_n^{(3)}  v_6(2^jx-k-n),
$$
where
$$
v_6(t)=\frac{1}{7200}
\begin{cases}
t^5, & 0\leq t \leq 1/2, \\
1-10t+40t^2-80t^3+80t^4-31t^5, & 1/2<t \leq 1, \\
-236+1175t-2330t^2+2290t^3-1105t^4+206t^5, & 1<t\leq 3/2, \\
6082-19885t+25750t^2-16430t^3+5135t^4-626t^5, & 3/2<t\leq 2, \\
3(-15914+38225t-36270t^2+16950t^3+3895t^4+352t^5), & 2<t\leq 5/2, \\
3(52836-99275t+73730t^2-27050t^3+4905t^4-352t^5), & 5/2<t\leq 3, \\
-250218+383385t-232950t^2+70230t^3-10515t^4+626t^5, &3<t \leq 7/2, \\
186764-240875t+123770t^2-31690t^3+4045t^4-206t^5, & 7/2<t\leq 4, \\
-55924+62485t-27910t^2+6230t^3-695t^4+31t^5, & 4<t\leq 9/2, \\
3125-3125t+1250t^2-250t^3+25t^4-t^5, & 9/2<t\leq 5,\\
0, & \text{otherwise}.
\end{cases}
$$

\bibliographystyle{siam}

\end{document}